%% file: main.tex
\documentclass[a4paper, 11pt]{article}
\input{preamble.tex}
\usepackage{bbm} 

\newcommand*{\LRP}[1]{\left(#1\right)}
\newcommand*{\SN}[1]{\left|{#1}\right|}      
\newcommand{\ONBQSigma}{\mathbf U_{\Qmat}} 
\newcommand{\ONBPSigma}{\mathbf U_{\Pmat}} 

\providecommand*{\N}[1]{\left\|{#1}\right\|} 
\providecommand*{\textCovv}[1]{\operatorname{Cov}({#1})}   
\providecommand*{\textVar}[1]{\operatorname{Var}({#1})}   
\providecommand*{\textN}[1]{\|{#1}\|} 
\newcommand*{\textSN}[1]{|{#1}|}      
\providecommand*{\Covv}[1]{\operatorname{Cov}\left({#1}\right)}   
\providecommand*{\Var}[1]{\operatorname{Var}\left({#1}\right)}   
\providecommand{\rank}{\operatorname{rank}}                        
\providecommand{\argmin}{\operatorname*{argmin}}  
\providecommand{\tr}{\operatorname{tr}}
\providecommand{\bbP}{\mathbb{P}}
\providecommand{\bbR}{\mathbb{R}}
\providecommand{\bbS}{\mathbb{S}}
\providecommand{\bbE}{\mathbb{E}}
\providecommand{\bbN}{\mathbb{N}}
\providecommand{\msf}{\mathsf{m}}
\providecommand{\CO}{{\cal O}}
\providecommand{\CN}{{\cal N}}
\providecommand{\CM}{{\cal M}}
\providecommand{\CS}{{\cal S}}
\providecommand{\CR}{{\cal R}}
\providecommand{\Diag}[1]{\operatorname{Diag}\left({#1}\right)}
\providecommand*{\dist}[2]{\operatorname{dist}({#1};{#2})}   

\newcommand{\VSigma}{{\mathbf\Sigma}}
\newcommand{\VDelta}{{\mathbf\Delta}}
\providecommand*{\Span}{\operatorname{span}}     
\providecommand{\Id}{\mathbf{Id}}
\newcommand*{\EUSP}[2]{\left<{#1},{#2}\right>} 

\renewcommand{\Im}{\operatorname{Im}}             
\newcommand{\Amat}{{\mathbf A}}
\newcommand{\Bmat}{{\mathbf B}}
\newcommand{\Gmat}{{\mathbf G}}
\newcommand{\Pmat}{{\mathbf P}}
\newcommand{\Qmat}{{\mathbf Q}}
\newcommand{\Hmat}{{\mathbf H}}
\newcommand{\Umat}{{\mathbf U}}
\newcommand{\Vmat}{{\mathbf V}}
\newcommand{\Smat}{{\mathbf S}}
\newcommand{\Mmat}{{\mathbf M}}
\newcommand{\tempMatOne}{\Gmat}
\newcommand{\tempMatTwo}{\Hmat}
\newcommand{\centerRV}[1]{\tilde{#1}}
\newcommand{\PrSigma}{\Pmat_{\VSigma}}
\newcommand{\QrSigma}{\Qmat_{\VSigma}}
\newcommand{\Amod}{\Amat_{\VSigma}}
\newcommand{\Bmod}{\Bmat_{\VSigma}}
\newcommand{\bbs}{{\mathbf{{b}}}}
\newcommand{\abs}{{\mathbf{{a}}}}
\newcommand{\dbs}{{\mathbf{{d}}}}
\newcommand{\ubs}{{\mathbf{{u}}}}
\newcommand{\vbs}{{\mathbf{{v}}}}
\newcommand{\rbs}{{\mathbf{{r}}}}
\newcommand{\xbs}{{\mathbf{{x}}}}
\newcommand{\ybs}{{\mathbf{{y}}}}
\newcommand{\Vmu}        {\boldsymbol{\upmu}}
\newcommand{\Vxi}        {\boldsymbol{\upxi}}
\makeatother

\title{
		\usefont{OT1}{bch}{b}{n}
		\huge Estimating covariance and precision matrices along subspaces\\
}

\date{}
\author[1]{\v Zeljko Kereta\thanks{Email: \texttt{zeljko@simula.no}} }
\author[1]{Timo Klock\thanks{Email: \texttt{timo@simula.no}} }

\affil[1]{Simula Research Laboratory, Machine Intelligence Department, Oslo, Norway}

\begin{document}
\maketitle

\begin{abstract}
We study the accuracy of estimating the covariance and the precision matrix of a $D$-variate sub-Gaussian distribution along a prescribed subspace or direction using the finite sample covariance. Our results show that the estimation accuracy depends almost exclusively on the components of the distribution that correspond to desired subspaces or directions. This is relevant and important for problems where the behavior of data along a lower-dimensional space is of specific interest, such as dimension reduction or structured regression problems.  We also show that estimation of precision matrices is almost independent of the condition number of the covariance matrix. The presented applications include direction-sensitive eigenspace perturbation bounds, relative bounds for the smallest eigenvalue, and the estimation of the single-index model. For the latter, a new estimator, derived from the analysis, with strong theoretical guarantees and superior numerical performance is proposed.
\end{abstract}


\section{Introduction}
\label{sec:introduction}
Estimating the covariance $\VSigma=\bbE(X-\bbE X)(X-\bbE X)^\top$ and the precision matrix $\VSigma^\dagger$ of a random vector $X \in \bbR^D$ is a
standard and long standing problem in multivariate statistics with applications in a number of mathematical and applied fields.
Notable examples include any form of dimension reduction, such as principal component analysis, nonlinear dimension reduction,
manifold learning, but also problems ranging from classification, regression, and signal processing to econometrics, brain imaging and social networks.

Given independent copies $X_1,\ldots,X_N$ of $X$, the most widely used estimator is the sample covariance $\hat \VSigma := \frac{1}{N}\sum_{i=1}^{N}X_i X_i^\top$, and the inverse thereof.
The crucial question in estimating covariance and precision matrices is to quantify the minimal number of samples $N$ ensuring that for a desired accuracy $\varepsilon > 0$, and a confidence level $u > 0$, we have
\begin{align}
\label{eq:covariance_estimation_bound}
\big\|{\hat \VSigma - \VSigma}\big\|_2 \leq \varepsilon S_\VSigma(X),\quad
\textrm{respectively},\quad
\big\|{\hat \VSigma^{\dagger} - \VSigma^{\dagger}}\big\|_2 \leq \varepsilon S_{\VSigma^\dagger}(X),
\end{align}
with probability at least $1-\exp(-u)$.
Constants $S_\VSigma(X)$ and $S_{\VSigma^\dagger}(X)$ in \eqref{eq:covariance_estimation_bound}
describe the dependence of the error with respect to
the distribution of $X$ and properties of $\VSigma$, respectively $\VSigma^\dagger$.

In practice however, we are often not directly interested in matrices $\VSigma$ or $\VSigma^{\dagger}$ themselves, but rather in
how they, as (bi)-linear operators, act on specific vectors or matrices.
In terms of concentration inequalities, this can be interpreted as developing bounds for
$\Amat(\hat \VSigma - \VSigma)\Bmat^\top$ and $\Amat(\hat \VSigma^{\dagger} - \VSigma^{\dagger})\Bmat^\top$, where $\Amat$ and $\Bmat$ are a given pair of (rectangular) matrices.
Bounds of this type in case of sub-Gaussian distributions are the principal subjects of this work.

For any submultiplicative matrix norm $\textN{\cdot}$  perturbations $\textN{\Amat(\hat \VSigma - \VSigma)\Bmat^\top}$ and $\textN{\Amat(\hat \VSigma^{\dagger} - \VSigma^{\dagger})\Bmat^\top}$
are bounded by $\textN{\Amat}\textN{\hat \VSigma - \VSigma}\textN{\Bmat}$ and $\textN{\Amat}\textN{\hat \VSigma^{\dagger} - \VSigma^{\dagger}}\textN{\Bmat}$.
This suggests that standard bounds for $\textN{\hat \VSigma - \VSigma}$ and $\textN{\hat \VSigma^{\dagger} - \VSigma^{\dagger}}$  (see the overviews in Sections \ref{subsec:state_of_the_art_covariance_estimation} and \ref{subsec:soa_precision_matrix}) suffice so long as the distribution is nearly isotropic, i.e. $\VSigma \approx \Id_D$.
However, many modern data analysis tasks explicitly rely on anisotropic distributions
because different spectral modalities of the covariance matrix provide crucial, and complementary, information about
the task at hand.
In this case, using norm submultiplicativity and standard bounds for
$\textN{\hat \VSigma - \VSigma}$ and $\textN{\hat \VSigma^{\dagger} - \VSigma^{\dagger}}$ overestimates incurred errors
because it decouples $\Amat$ and $\Bmat$ from their effect on covariance and precision matrices.
Thus, such bounds cannot capture the true behavior of the estimation error.

A typical example that leverages different modalities of (conditional) covariance matrices
are problems that analyze the structure of point clouds, such as manifold learning.
This is because such methods are often prefaced by a linearization step, where the globally non-linear geometry is locally approximated by tangential spaces.
In such a case the conditional covariance of localized data points is anisotropic, i.e. eigenvalues in tangential directions are notably larger than those in non-tangential directions, and the degree of anistropicity increases as the data is more localized \cite{little15multiscaleSVD}, facilitating the linearization.

Anisotropic distributions also play an important role in high-dimensional nonparametric regression problems
that use structural assumptions. Denoting $Y$ as the dependent output variable,
a popular example is the multi-index model $\bbE[Y|X=\xbs] = g(\Amat^\top \xbs)$, for a matrix
$\Amat$ with $\rank(\Amat) \ll D$.
The complexity of the underlying nonparametric regression problem can be significantly
reduced by first identifying $\Im(\Amat)$ and then performing nonparametric regression in
$\bbR^{\rank(\Amat)}$. Many methods in the structural dimension reduction literature \cite{adragni2009sufficient,ma2013review}
propose estimators for $\Amat$ that compare the global covariance matrix $\VSigma$ (typically assumed to be isotropic),
with conditional covariance matrices $\Covv{X|Y \in \CR}$.
Here $\CR \subset \Im(Y)$ represents a connected level set of the output.
The rationale behind this approach
is that conditioning breaks the isotropicity and induces different spectral modalities
with respect to directions belonging to $\Im(\Amat)$ and its orthogonal complement. This is
then leveraged to identify $\Im(\Amat)$ \cite{dennis2000save,li2007directional}.

We showcase the usability of bounds for $\Amat(\hat \VSigma^{\dagger} - \VSigma^{\dagger})\Bmat^\top$
on a concrete example in Section \ref{sec:application} by analyzing the ordinary
least squares estimator $\VSigma^{\dagger}\Covv{X,Y}$ as an estimator for the index
vector $\abs$ in the single-index model $\bbE[Y|X=\xbs] = g(\abs^\top \xbs)$.
Our analysis extends studies \cite{brillinger1982generalized, balabdaoui2019score} and shows how is the accuracy of the estimator be affected
by anisotropicity of $X$. Furthermore, an examination of developed estimation bounds motivates a modification of the ordinary least squares estimator
via an approach based on conditioning and averaging in the spirit of structural dimension reduction techniques \cite{adragni2009sufficient,ma2013review}.
We validate the superiority of this modified estimator theoretically and
numerically, and thereby show how a careful tracking of different modalities of the distribution $X$ helps to develop improved
 methods for common data analysis tasks.

Bounds developed in this work also have a few immediate corollaries, which
might be of independent interest. These include eigenspace perturbation bounds similar to
\cite[Theorem 1]{yu2014useful}, but which are sensitive to the behavior of $X$
in the direction corresponding to the eigenspace of interest, and a relative bound
for the smallest eigenvalue of $\hat \VSigma$ comparable to \cite[Theorem 2.2]{yaskov2014lower}, but without
the isotropicity assumption.

\subsection{State of the art: covariance matrix estimation}
\label{subsec:state_of_the_art_covariance_estimation}
The most common bounds for estimating the covariance matrix from finitely many observations
consider sub-Gaussian \cite{vershynin2018high,vershynin2012close} and bounded \cite{rosasco2010learningwithintegraloperator} random vectors. They state that with probability at least $1-\exp(-u)$
\begin{align}
\label{eq:subGaussian_X}
&\textN{\hat \VSigma - \VSigma}_2 \lesssim  \textN{X}_{\psi_2}^2\LRP{\sqrt{\frac{D+u}{N}}\vee\frac{D+u}{N}},\,\\
\label{eq:bounded_X}
&\textN{\hat \VSigma - \VSigma}_F   \lesssim  C_X^2\sqrt{\frac{u}{N}},\, \, \textrm{ provided }  \textN{X}_2 \leq C_X\ a.s,
\end{align}
where $A \lesssim B$ means $A\leq CB$ for some universal constant $C$.
Besides these two cases, researchers have over the years investigated minimal
moment- or tail conditions on $X$ such that bounds similar to \eqref{eq:subGaussian_X}
can be achieved. We refer to papers \cite{vershynin2012close,srivastava2013covariance,adamczak2010quantitative,vershynin2010introduction}
that consider more general classes of distributions.
The most general setting we are aware of is \cite{srivastava2013covariance}
that considers distributions which for universal $C, \eta > 0$ satisfy the tail condition
\begin{align*}
\bbP\big(\N{\Pmat X}_2^2 > t\big)\leq Ct^{-1-\eta},\quad \textrm{for } t > C\rank(\Pmat),
\end{align*}
for every orthogonal projection $\Pmat$.
Distributions satisfying this condition include log-concave random variables (e.g. uniform distributions on convex sets)
and product distributions, where the marginals have uniformly bounded $4+s$ moments
for some $s > 0$.

Our bounds show that the sample covariance estimator $\hat \VSigma$ automatically and implicitly
adapts to $\rank(\hat \VSigma)$, which serves as a complexity parameter of the estimation problem.
However, in the regime $N<\rank(\VSigma)$ the sample covariance is rank deficient and the estimation is in general 	not possible.
Instead, structural assumptions, such as sparsity or bandedness, are needed to reduce the effective complexity of the problem and allow consistent estimation.
These assumptions can be leveraged by regularized estimation techniques, which include banding \cite{BickelLevinaBanding}, thresholding \cite{cai2016estimating,BickelLevinaThresholding}, or penalized likelihood estimation \cite{PenalisedLikelihoodCovs}.
We refer to \cite{fan2016overview,cai2016estimating} for detailed reviews of existing methods.

We also mention \cite{koltchinskii2017concentration} that considers
Gaussian random variables taking values in general Banach spaces, and \cite{koltchinskii2017normal} which
analyze the concentration of spectral projectors of the covariance matrix, and (bi)-linear forms
thereof, for Gaussian random variables in a separable Hilbert space. This is conceptually related
to our work as the bounds take into account information about relevant spectral projectors, and their
interplay with vectors to which they are applied. This has recently been extended to the
estimation of smooth functionals of covariance operators \cite{koltchinskii2017asymptotically, koltchinskii2020efficient}
for Gaussians in separable Hilbert spaces.

\subsection{State of the art: precision matrix estimation}
\label{subsec:soa_precision_matrix}
Estimation of the precision matrix is relevant for many problems, ranging from simple tasks
such as data transformations (e.g. standardization $Z := \VSigma^{-1/2}X$), to  applications
that include linear discriminant analysis, graphical modeling, or
complex data visualization.
Furthermore,  precision matrix encodes information about partial correlations of features of $X$. Namely, if $X$ follows a Gaussian (or paranormal)
distribution, the $ij$-th entry of $\VSigma^{\dagger}$ is zero if the $i$-th and the $j$-th
feature are conditionally independent.

The inverse $\hat\VSigma^\dagger$ of the sample covariance, constructed from $N$ independent copies of a mean zero random vector $X \in \bbR^D$, is a well-behaved estimator of $\VSigma^\dagger$ as $N\rightarrow\infty$
and $D$ is considered fixed \cite{AndersonStatAn}.
In such a case bounds for the precision matrix can be obtained by using general perturbation bounds for the Moore-Penrose inverse.
One of the first such bounds \cite{wedin1973perturbation} states that for $\tempMatOne \in \bbR^{d_1 \times d_2}$, and an additively perturbed matrix $\tempMatTwo = \tempMatOne + \VDelta$, we have
\begin{align}
\label{eq:Wedins_MP_bound}
\textN{\tempMatTwo^{\dagger} - \tempMatOne^{\dagger}} &\leq \omega \max\big\{\textN{\tempMatOne^{\dagger}}_2^2,\textN{\tempMatTwo^{\dagger}}_2^2\big\}\textN{\VDelta},\text{ and }\\
\label{eq:Wedins_MP_bound_rank_equality}
\textN{\tempMatTwo^{\dagger} - \tempMatOne^{\dagger}} &\leq \omega \textN{\tempMatOne^{\dagger}}_2\textN{\tempMatTwo^{\dagger}}_2\textN{\VDelta},\quad \textrm{if } \rank(\tempMatOne) = \rank(\tempMatTwo),
\end{align}
where $\textN{\cdot}$ is any unitarily invariant norm, and $\omega$ is a small universal constant \cite{meng2010optimal}.
Recent studies \cite{li2018some,xu2018perturbation} examine the influence of the
perturbation in greater detail, implying  the bound
\begin{align*}
\textN{\tempMatTwo^{\dagger} - \tempMatOne^{\dagger}}_F &\leq \min\left\{\textN{\tempMatTwo^{\dagger}}_2 \textN{\tempMatOne^{\dagger}\VDelta}_F, \textN{\tempMatOne^{\dagger}}_2\textN{\tempMatTwo^{\dagger}\VDelta}_2\right\},\\ &\textrm{if } \rank(\tempMatOne) = \rank(\tempMatTwo) = \min\{d_1,d_2\}.
\end{align*}
Using these general perturbation bounds it is easy to derive first concentration bounds for the precision matrix.
For example, assume $X$ is sub-Gaussian and that the number of independent data samples satisfies $N =C4^{k}(D+u)\N{X}_{\psi_2}^4/\lambda_{\rank(\VSigma)}(\VSigma)^2$ for some universal $C>0$ and $k\in\bbN$.
Equation \eqref{eq:subGaussian_X} and Weyl's bound \cite{weyl1912asymptotische} imply
$$\lambda_{\rank(\VSigma)}(\hat \VSigma) \geq \lambda_{\rank(\VSigma)}(\VSigma) - \textN{\hat \VSigma - \VSigma}_2 \geq \lambda_{\rank(\VSigma)}(\VSigma)(1-2^{-k}),$$
and consequently $\textN{\hat \VSigma^{\dagger}}_2 \leq (1-2^{-k})^{-1}\textN{\VSigma^{\dagger}}_2$.
The perturbation bound \eqref{eq:Wedins_MP_bound_rank_equality}
and covariance bound \eqref{eq:subGaussian_X} give with probability $1-\exp(-u)$
\begin{align}
\label{eq:generalized_inverse_bound}
\textN{\hat \VSigma^{\dagger}-\VSigma^{\dagger}}_2 \lesssim \textN{\VSigma^{\dagger}}_2^2 \textN{\hat \VSigma - \VSigma}_2 \lesssim \textN{\VSigma^\dagger}_2^2 \N{X}_{\psi_2}^2\sqrt{\frac{D+u}{N}},
\end{align}
where the higher order term in \eqref{eq:subGaussian_X} vanishes in the applicable regime
$N \geq C(D+u)\N{X}_{\psi_2}^4/\lambda_{\rank(\VSigma)}^2(\VSigma)$.
While this effectively provides bounds as soon as the covariance
perturbation is bounded, in this work we show that \eqref{eq:generalized_inverse_bound} overestimates the error by assigning a quadratic scaling of $\textN{\VSigma^\dagger}_2$
(Corollary \ref{cor:condition_free_precision_matrix_estimation}).

Moreover, we are not aware of precision matrix bounds that take into account the specific nature of
the perturbation
nor of bounds for
$\textN{\Amat(\tempMatTwo^\dagger - \tempMatOne^\dagger)\Bmat}_2$ that are sensitive to objects of interest.

If $\rank(\VSigma)$ grows with $N$, then $\hat\VSigma$ is not a consistent estimator
of $\VSigma$ and thus the precision matrix cannot be estimated well by inverting the sample covariance matrix.
Various families of structured precision
matrices have been studied to mitigate these issues, motivated by applications in genomics, finance, and other fields.
Dominant assumptions are sparsity and bandedness, which are exploited
through the use of regularized estimators. Algorithms for estimating $\VSigma^{\dagger}$ based on
regularization include computing $\VSigma^\dagger$ column by column through entry-wise Lasso
\cite{meinshausenGLASSO,FHTGLasso}, constrained $\ell_1$ minimization \cite{TCaiCLIME},
 adaptive $\ell_1$ minimization \cite{ACLIME}, $\ell_1$ regularized score matching
 \cite{FASP}, or ridge regressors \cite{WesselRIDGE}. See \cite{fan2016overview,cai2016estimating}
 for comprehensive overviews.

\subsection{Overview and contributions}
\label{subsec:contributions}
Throughout, $X \in \bbR^D$ is a sub-Gaussian random vector with $\centerRV{X} := X - \bbE X$ and $\VSigma = \Covv{X}$,
and $X_1,\ldots,X_N$ are independent copies of $X$.
Sub-Gaussians are a broad class of light-tailed distributions that have received increasing attention in recent years and are used in many branches of probability and statistics.
We define finite sample estimators of $\bbE X$ and $\VSigma$ by
$\hat \Vmu_X := N^{-1}\sum_{i=1}^{N}X_i$ and $\hat \VSigma := N^{-1}\sum_{i=1}^{N}(X_i-\hat \Vmu_X) (X_i-\hat \Vmu_X)^\top$.

Let $\Amat \in \bbR^{d_1\times D}$, $\Bmat \in \bbR^{d_2\times D}$ be matrices determining a direction, subspace, or generally an object of interest.
We can summarize our findings as follows.
\begin{enumerate}[label = (\arabic*)]
\item In Section \ref{sec:covariance_estimation} we show that with probability at least $1-\exp(-u)$
\begin{align}
\label{eq:contributions_covariance_matrix_bound_P}
\textN{\Amat(\hat \VSigma - \VSigma)\Bmat^\top}_2 \lesssim\textN{\Amat\centerRV{X}}_{\psi_2}\textN{\Bmat \centerRV{X}}_{\psi_2}\msf\left(\sqrt{\frac{d_A + d_B +u}{N}}\right),
\end{align}
where $\msf(t) = t \vee t^2$, $d_A := \rank(\Amat\VSigma)$, and $d_B := \rank(\Bmat\VSigma)$.
This is similar to \cite[Proposition 2.1]{vershynin2012close} but replaces the sub-Gaussian norm $\textN{\centerRV{X}}_{\psi_2}$
by direction/subspace dependent quantities $\| \Amat \centerRV{X}\|_{\psi_2}$ and $\| \Bmat \centerRV{X}\|_{\psi_2}$.
\item In Section \ref{sec:precision_matrix_estimation} we show  that with probability at least $1-\exp(-u)$,
we have
\begin{align}
\label{eq:contributions_precision_matrix_bound_P}
\textN{\Amat(\hat \VSigma^{\dagger} - \VSigma^{\dagger})\Bmat^\top}_2 \lesssim \textN{\Amat\VSigma^{\dagger}\centerRV{X}}_{\psi_2}\textN{\Bmat\VSigma^{\dagger}\centerRV{X}}_{\psi_2}\sqrt{\frac{\rank(\VSigma)+u}{N}},
\end{align}
provided $N \geq C \textN{\sqrt{\VSigma^{\dagger}}\centerRV{X}}_{\psi_2}^4$, for an absolute constant $C>0$,
where $\sqrt{\VSigma^{\dagger}}\centerRV{X}$ is the standardization of a random variable $X$, with $\textCovv{\sqrt{\VSigma^{\dagger}}\centerRV{X}} = \Id_D$.
\item In Section \ref{sec:precision_matrix_estimation} we show stronger bounds for \eqref{eq:contributions_covariance_matrix_bound_P} and \eqref{eq:contributions_precision_matrix_bound_P} in case of bounded random vectors.
\end{enumerate}

\begin{remark}
\label{rem:initial_remark}
As we point out in Corollary \ref{cor:condition_free_precision_matrix_estimation},
the bound \eqref{eq:contributions_precision_matrix_bound_P} is  interesting even
when $\Amat = \Bmat = \Id_D$, particularly in view of  general perturbation
bounds such as \eqref{eq:generalized_inverse_bound}.
Consider for example a random vector $X$ for which the sub-Gaussian norm is a good proxy for the variance, i.e.
so that $\| \VSigma^{\dagger}\centerRV{X}\|_{\psi_2} \approx \sqrt{\| \VSigma^{\dagger}\|_2}$ holds.
(This is the case for example if $X \sim \CN(\Vmu,\VSigma)$, or more generally for strict sub-Gaussians \cite{arbel2019strictsubgaussian}).
The right hand side of \eqref{eq:contributions_precision_matrix_bound_P} then scales linearly
in $\| \VSigma^{\dagger}\|_2$, whereas \eqref{eq:generalized_inverse_bound}
shows a quadratic behavior. This has a significant impact for ill-conditioned
covariance matrices and implies that inverting the sample covariance exhibits better conditioning than inverting a general matrix perturbation.
The same effect is observed if $\Amat$ and $\Bmat$ are arbitrary.
\end{remark}

Two applications of bounds \eqref{eq:contributions_covariance_matrix_bound_P} and \eqref{eq:contributions_precision_matrix_bound_P}
are presented. In Section \ref{sec:covariance_estimation} we use the covariance bound
\eqref{eq:contributions_covariance_matrix_bound_P} to establish a bound for perturbations of eigenspaces of the covariance matrix that is sensitive to the distribution of the random vector in the eigenspace of interest.
This is relevant for example when estimating manifolds from unlabeled point cloud data,
see \cite{mitra2003estimating,ouyang2005normal,klasing2009comparison}.

In Section \ref{subsec:ordinary_SIM} we use
\eqref{eq:contributions_precision_matrix_bound_P} to establish sharp concentration bounds for single-index model estimation.
In this model a response $Y \in \bbR$ is assumed to follow the regression model
$\bbE[Y|X] = f(\abs^\top X)$, and the task is to estimate the unknown vector $\abs$ using a finite data set $\{(X_i,Y_i):i=1,\ldots,N\}$.
A common estimator is the normalized ordinary least squares vector, for which we provide
direction-sensitive concentration bounds.
Furthermore, our analysis yields an insight into how the estimator can be improved by a simple and straightforward procedure based on conditioning and averaging. This is presented in Section \ref{subsec:conditional_SIM}.

Most proofs are deferred to the Appendix for the sake of brevity and clarity.
\subsection{General notation}
\label{subsec:notation}
We denote $[M] = \{1,\ldots,M\}$, $a \wedge b = \min\{a,b\} $, $a \vee b = \max\{a,b\}$,
and we may use the auxiliary function $\msf(t) = t \vee t^2$.
For a real, symmetric matrix $\Amat \in \bbR^{d\times d}$ we denote by $\lambda_1(\Amat) \geq \ldots\geq \lambda_d(\Amat)$
the ordered set of its eigenvalues, and by $\ubs_1(\Amat),\ldots,\ubs_d(\Amat)$ the corresponding eigenvectors.
$\textN{\cdot}_2$ denotes the spectral norm of a matrix, and the Euclidean norm of a vector, and $\EUSP{\cdot}{\cdot}$ is the dot product.
$\textN{\cdot}_F$ is the Frobenius norm. $\Id_D\in\bbR^{D\times D}$ is the identity matrix.
$\bbS^{D-1}$ is the unit sphere in $\bbR^D$. For any random vector $X$
we write $\centerRV{X} := X - \bbE X$. For $p \geq 1$ and a random variable $X$ we define the Orlicz norm
\[
\textN{X}_{\psi_p} := \inf\{s > 0 : \bbE\exp(\SN{X/s}^p)\leq 2\}.
\]
The definition extends to random vectors $X \in \bbR^D$ by
\begin{equation}
\label{eqn:subG_vect_norm}
\textN{X}_{\psi_p} := \sup_{\vbs \in \bbS^{D-1}}\textN{\vbs^\top X}_{\psi_p} < \infty.
\end{equation}
We only use $p=1$ (sub-exponential) and $p=2$ (sub-Gaussian).
If $\Omega$ is a finite set, $\SN{\Omega}$ denotes its cardinality. If $\Omega$
is an interval, $\SN{\Omega}$ denotes its length.
Throughout the paper $C$ is a placeholder for a positive universal constant that may have a different value on each occurrence, even in the same line of text.
Furthermore, we sometimes use $A\lesssim B$ instead of $A\leq CB$.

\section{Covariance matrix estimation}
\label{sec:covariance_estimation}
In this section we present bounds for covariance  and eigenspace estimation that are sensitive to the distribution of a given random vector in directions of interest.
The following matrix concentration bound is the fundamental tool of our analysis.
\begin{lemma}
\label{lem:directed_covariance_estimation}
Let $\Amat \in \bbR^{d_1 \times D}, \Bmat\in \bbR^{d_2 \times D}$ and
define $d_A=\rank(\Amat\VSigma)$, and $d_B=\rank(\Bmat\VSigma)$.
Then for all $u>0$, with probability at least $1 - \exp(-u)$ we have
for $\msf(t) = t \vee t^2$
\begin{equation}
\label{eq:directed_covariance_estimation_subGaussian}
\textN{\Amat(\hat \VSigma - \VSigma)\Bmat^\top\!}_2 \lesssim\!\textN{\Amat\centerRV{X}}_{\psi_2}\textN{\Bmat\centerRV{X}}_{\psi_2}\msf\left(\sqrt{\frac{d_A+d_B+u}{N}}\right).
\end{equation}
\end{lemma}

\begin{remark}\label{rem:bounded}
An analogous result holds for almost surely bounded random vectors by using a different concentration inequality.
In that case $\textN{\cdot}_{\psi_2}$ can be replaced by a bound for the Euclidean norm of $X$, and the dimensionality
does not appear in the requirement on $N$. We will return to this point at the end of Section \ref{sec:precision_matrix_estimation}.
\end{remark}

The proof of Lemma \ref{lem:directed_covariance_estimation} by and large follows along the lines of traditional concentration results.
Indeed, if $\Amat=\Bmat$, the result would follow by applying \cite[Proposition 2.1]{vershynin2012close} to the random vector $\Amat X$, along with some minor adjustments to account for the ranks.
When $\Amat\neq\Bmat$, a somewhat more careful tracking of the behavior of the random vector $X$, with respect directions induced by matrices $\Amat$ and $\Bmat$ is needed.
In the end, as \eqref{eq:directed_covariance_estimation_subGaussian} suggests, the payoff is that the error rate scales only with components of $X$ along those directions.

\begin{remark}
Some works that consider concentration inequalities for sub-Gaussian random variables do so with respect to the \emph{effective rank}, defined as $r(\VSigma) = \tr{(\VSigma)}/\N{\VSigma}_2$, see for instance \cite{koltchinskii2017asymptotically} or \cite[Remark 5.6.3]{vershynin2018high}.
Effective rank cannot exceed the true rank of a matrix, and unlike $\rank(\VSigma)$, it is less affected by small eigenvalues.
It can thus be a useful surrogate for approximately low dimensional distributions, and can in those cases be used to provide informative estimation bounds
even if $N \ll \rank(\VSigma)$.

On the other hand, $r(\VSigma)$ is not as useful for precision matrix estimation, since inverting a matrix reverses the ordering of the eigenvalues. Thus, $r(\VSigma)$ should be replaced with $r(\VSigma^\dagger)$.
Furthermore, our current proof technique for precision matrix estimation
requires $\Im(\hat \VSigma) = \Im(\VSigma)$, which immediately implies $N \geq \rank(\VSigma)$.
To provide a unified framework, we decided to abstain from using the effective rank in our results.

\end{remark}

Applying Lemma \ref{lem:directed_covariance_estimation} we can reconstruct known error rates in case of low-rank  distributions in high-dimensional ambient spaces.
\begin{corollary}\label{cor:low_rank}
For all $u>0$  with probability at least $1 - \exp(-u)$ we have
\begin{equation}
\label{eq:low_rank_estimation}
\textN{\hat \VSigma - \VSigma}_2 \lesssim \textN{\centerRV{X}}_{\psi_2}^2 \msf\LRP{\sqrt{\frac{\rank(\VSigma)+u}{N}}},\quad \textrm{where } \msf(t) = t \vee t^2.
\end{equation}
\end{corollary}

Lemma \ref{lem:directed_covariance_estimation} also has an immediate effect on the estimation of eigenvectors and eigenspaces.
Denote by $\Pmat_{i,l}(\VSigma) := \sum_{k=i}^{l}\ubs_k(\VSigma)\ubs_k(\VSigma)^\top$ the orthoprojector onto the space spanned by $i$-th to $l$-th eigenvectors of $\VSigma$, and let
$\Pmat_{i,l}(\hat \VSigma)$ be the corresponding finite sample estimator.
Moreover, denote $\Qmat_{i,l}(\hat \VSigma) := \Id_D - \Pmat_{i,l}(\hat \VSigma)$,
and $\dist{I_1}{I_2} := \inf_{t \in I_1, t' \in I_2}\SN{t-t'}$ for $I_1,\, I_2 \subset \bbR$.
\begin{proposition}
\label{thm:directed_eigenspace_perturbation}
Let $i \leq l \in \bbN$, and define
\begin{equation}
\label{eqn:directional_DK_spgap}
\begin{aligned}
\delta_{il} &= \dist{[\lambda_l(\VSigma),\lambda_i(\VSigma)]}{[-\infty,\lambda_{l+1}(\hat\VSigma)]\cup[\lambda_{i-1}(\hat\VSigma),+\infty]}, \\
&\qquad\text{ with } \lambda_0(\hat \VSigma) := \infty,\, \lambda_{D+1}(\hat \VSigma) = -\infty.
\end{aligned}
\end{equation}
For any $u>0$, with probability at least $1 - \exp(-u)$ we have, with $\msf(t) = t \vee t^2$,
\begin{align}\label{eqn:directional_DK}
\textN{\Qmat_{i,l}(\hat \VSigma)\Pmat_{i,l}(\VSigma)}_2 \lesssim  \frac{\textN{\Pmat_{i,l}(\VSigma)\centerRV{X}}_{\psi_2}\big\|{\centerRV{X}}\big\|_{\psi_2}}{\delta_{il}}\msf\LRP{\sqrt{\frac{\rank(\VSigma)+u}{N}}}.
\end{align}
\end{proposition}

\begin{proof}
Davis-Kahan Theorem in \cite[Theorem 7.3.2]{bhatia2013matrix} gives
\begin{align*}
\textN{\Qmat_{i,l}(\hat \VSigma)\Pmat_{i,l}(\VSigma)}_2 &\leq
\frac{\pi}{2}\frac{\textN{\Qmat_{i,l}(\hat \VSigma)(\VSigma - \hat \VSigma)\Pmat_{i,l}(\VSigma)}_2}{\delta_{il}} \leq \frac{\pi}{2}\frac{\textN{(\VSigma - \hat \VSigma)\Pmat_{i,l}(\VSigma)}_2}{\delta_{il}}.
\end{align*}
The claim now follows by applying Lemma \ref{lem:directed_covariance_estimation} with $\Amat = \Id_D$ and $\Bmat = \Pmat_{i,l}(\VSigma)$.
\end{proof}

Typical bounds for eigenspace perturbations $\textN{\Qmat_{i,l}(\hat \VSigma)\Pmat_{i,l}(\VSigma)}_2$ take the specific eigenspace into account only through the denominator, whereas the numerator relies on squared terms of the form $\textN{\centerRV{X}}_{\psi_2}^2$ in the sub-Gaussian case, or a bound for $\textN{{X}}_{2}^2$ in the bounded case.
Expression \eqref{eqn:directional_DK} is thus beneficial if $\textN{\Pmat_{i,l}(\VSigma)\centerRV{X}}_{\psi_2}$ is smaller than $\textN{\centerRV{X}}_{\psi_2}$, as it
provides a sharper estimate.

In order to ensure $\delta_{il} > 0$, the covariance matrix $\VSigma$ must have a population eigengap, that is, $\delta_{il}^* := (\lambda_{i-1}(\VSigma)-\lambda_i(\VSigma)) \wedge (\lambda_l(\VSigma) - \lambda_{l+1}(\VSigma)) > 0$,
and sufficiently many samples are required in order to stabilize $\delta_{il}$ around $\delta^*_{il}$.
The latter is typically achieved by first using Weyl’s bound \cite{weyl1912asymptotische},
giving $\textSN{\lambda_j(\hat \VSigma) -\lambda_j(\VSigma)}\leq \textN{\hat \VSigma-\VSigma}_2$ for all $j \in [D]$,
and then applying a concentration bound for $\textN{\hat \VSigma-\VSigma}_2$. A consequence however
is that Proposition \ref{thm:directed_eigenspace_perturbation} is only informative if we have
sufficiently many samples with respect to $\delta^*_{il}$ and $\textN{\centerRV{X}}_{\psi_2}$.
Thus, the estimation error is no longer sensitive to the eigenspace of interest.

Preserving the dependence on $\textN{\Pmat_{il}(\VSigma)\centerRV{X}}_{\psi_2}$, instead of on $\textN{\centerRV{X}}_{\psi_2}$,
requires the use of relative eigenvalue bounds. Such bounds have been recently provided in \cite{reiss2020nonasymptotic}
under the assumption that $X$ is strictly sub-Gaussian \cite{arbel2019strictsubgaussian},
i.e. for some $K > 0$, and for arbitrary matrices $\Umat\in\bbR^{k\times D}$, $X$ satisfies
\begin{equation}\label{eqn:norm_var_proxy_new}
\|\Umat\centerRV{X}\|_{\psi_2}^2 \leq K \| \textCovv{\Umat\centerRV{X}} \|_2.
\end{equation}
This asserts that the sub-Gaussian norm is a good proxy
for variances of one-dimensional marginals of the random vector $X$, and is for instance
satisfied for $X \sim \CN(\boldsymbol{0},\VSigma)$ with $K = 1$. For such distributions we obtain the following
concentration of the sample eigengap.

\begin{lemma}
\label{lem:concentration_sample_eigengap}
Assume $X$ satisfies \eqref{eqn:norm_var_proxy_new} for some $K > 0$.
Let $i,l\in \bbN$ with $1\leq i \leq l\leq \rank(\VSigma)$, and $\{i,l\} \neq \{1,\rank(\VSigma)\}$.
Let $\msf(t) = t \vee t^2$. There exists a constant $C_K>0$, depending only on $K$, so that whenever
\begin{align}
\label{eqn:req_N_concentration_sample_eigengap}
N > C_K\left(\rank(\VSigma)\vee u\right)\msf\left(\frac{\lambda_{l+1}(\VSigma)}{\lambda_l(\VSigma)-\lambda_{l+1}(\VSigma)} \vee \frac{\lambda_{i-1}(\VSigma)}{\lambda_{i-1}(\VSigma)-\lambda_i(\VSigma)}\right),
\end{align}
(with conventions $\lambda_{0}(\VSigma) = \infty$, $\lambda_{\rank(\VSigma)+1}(\VSigma) = 0$, and $\infty/\infty = 0$)
for any $u>0$ with probability at least $1 - 2\exp(-u)$ we have $\delta_{il} \geq \delta_{il}^*/2$.
\end{lemma}

The case $\{i,l\} = \{1,\rank(\VSigma)\}$ is equivalent to asking whether $\Im(\hat \VSigma) = \Im(\VSigma)$.
This holds for $N \geq \rank(\VSigma)$ if the law of $X$ has a density which is absolutely
continuous with respect to the Lebesgue measure on $\Im(\VSigma)$, because $X_1,\ldots,X_N$ are almost surely linearly
independent.
It also holds with probability at least $1-\exp(-u)$ if $X$ is sub-Gaussian as soon as $N > C_K(\rank(\VSigma) + u)$,
provided \eqref{eqn:norm_var_proxy_new} holds, and if \eqref{eqn:norm_var_proxy_new} does not hold then we need
$N > C(\rank(\VSigma) + u)\textN{\sqrt{\VSigma^{\dagger}}\centerRV{X}}_{\psi_2}^4$.

Lemma \ref{lem:concentration_sample_eigengap} now allows to refine Proposition \ref{thm:directed_eigenspace_perturbation} with a population eigengap.
\begin{proposition}
\label{thm:directed_eigenspace_perturbation_pop_gap}
Assume $X$ satisfies \eqref{eqn:norm_var_proxy_new} for some $K > 0$. Let $1\leq i \leq l \leq \rank(\VSigma)$
with $\{i,l\} \neq \{1,\rank(\VSigma)\}$. Let $\msf(t) = t \vee t^2$.
There exists $C_K>0$, depending only on $K$, so that whenever
$N$ satisfies \eqref{eqn:req_N_concentration_sample_eigengap} for any $u>0$,
with probability at least $1 - \exp(-u)$ we have
\begin{align}\label{eqn:directional_DK_pop_gap}
\textN{\Qmat_{i,l}(\hat \VSigma)\Pmat_{i,l}(\VSigma)}_2 &\leq  C_K\frac{\textN{\Pmat_{i,l}(\VSigma)\centerRV{X}}_{\psi_2}\big\|{\centerRV{X}}\big\|_{\psi_2}}{\delta_{il}^*}\msf\LRP{\sqrt{\frac{\rank(\VSigma)+u}{N}}}\nonumber
\\&\leq C_K\frac{\sqrt{K\lambda_{i}(\VSigma)}\big\|{\centerRV{X}}\big\|_{\psi_2}}{\delta_{il}^*}\msf\LRP{\sqrt{\frac{\rank(\VSigma)+u}{N}}}
\end{align}
\end{proposition}
\begin{proof}
The first inequality follows immediately by first conditioning on events in Proposition \ref{thm:directed_eigenspace_perturbation}, Lemma \ref{lem:concentration_sample_eigengap},
and then using the union bound (probability $1-3\exp(-u)$ can be adjusted by adjusting $C_K$). The second inequality in \eqref{eqn:directional_DK_pop_gap}
comes from additionally using \eqref{eqn:norm_var_proxy_new} and $\textN{\Covv{\Pmat_{i,l}(\VSigma)X}}_2 = \lambda_i(\VSigma)$.
\end{proof}
Recently, \cite{yu2014useful} showed a useful alternative
\begin{align}
\label{eq:samford_bound}
\textN{\Qmat_{i,l}(\hat \VSigma)\Pmat_{i,l}(\VSigma)}_F \lesssim  \frac{\big({D^{1/2}\textN{\hat \VSigma-\VSigma}_2 \wedge \textN{\hat \VSigma - \VSigma}_F}\big)}{\delta^*_{il}}.
\end{align}

Thus, with \eqref{eq:samford_bound} we do not have to stabilize the sample eigengap, and the eigenspace perturbation bound \eqref{eq:samford_bound} can be
used for arbitrary $N\geq 1$.
A natural question to ask is whether Proposition \ref{thm:directed_eigenspace_perturbation}
holds if $\delta_{il}$ is replaced with $\delta^*_{il}$.
The following example strongly suggests this is not the case.
\begin{example}
\label{example:eigengap}
Assume that Proposition \ref{thm:directed_eigenspace_perturbation} holds with $\delta_{il}^\ast$ in place of $\delta_{il}$.
Let $X \sim \CN(\boldsymbol{0},\VSigma)$, where $\VSigma = \sum_{i=1}^{D-1}\ubs_i\ubs_i^\top + \eta^2 \ubs_D\ubs_D^\top$,
for $\eta < 1$.
Notice that in this case $\textN{\centerRV{X}}_{\psi_2} = 1$ and $\textN{\ubs_D^\top X}_{\psi_2} = \eta$, by the definition of the sub-Gaussian norm.
For any $N\geq 1$ with probability at least $1-\exp(-u)$ we would have
\begin{align*}
\textN{\Qmat_{D,D}(\hat \VSigma)\Pmat_{D,D}(\VSigma)}_2 \lesssim \frac{\eta}{1-\eta^2} \msf\bigg(\frac{D+u}{N}\bigg).
\end{align*}
In particular, we have an increasingly small estimation error as $\eta \rightarrow 0$ by using only one data sample, i.e. by estimating $\Pmat_{D,D}(\VSigma)$ based
on the eigendecomposition of a rank one matrix $XX^\top$.
\end{example}

\section{Precision matrix estimation}
\label{sec:precision_matrix_estimation}
In this section we investigate directional estimates of the precision matrix
$\VSigma^{\dagger}$ through the empirical precision matrix $\hat \VSigma^{\dagger}$,
analogously to results in Section \ref{sec:covariance_estimation}.

\begin{theorem}
\label{thm:condition_free_directional_precision_matrix_estimation_arbitrary}
Let $\Amat \in \bbR^{d_1\times D}$, $\Bmat \in \bbR^{d_2\times D}$.
There exists a uniform constant $C>0$ such that if $N > C(\rank(\VSigma) + u)\textN{\sqrt{\VSigma^{\dagger}}\centerRV{X}}_{\psi_2}^4$
for any $u>0$ with probability at least  $1-\exp(-u)$ we have
\begin{align}
\label{eq:cross_term_result_arbitrary_A_B}
\|{\Amat({\hat \VSigma^{\dagger} - \VSigma^{\dagger}})\Bmat^\top}\|_2 \lesssim \textN{\Amat\VSigma^{\dagger} \centerRV{X}}_{\psi_2} \textN{\Bmat\VSigma^{\dagger} \centerRV{X}}_{\psi_2}\sqrt{\frac{\rank(\VSigma)+u}{N}}.
\end{align}
\end{theorem}

Let us comment on the implications of Theorem \ref{thm:condition_free_directional_precision_matrix_estimation_arbitrary}.
First, we note that $\textN{\sqrt{\VSigma^{\dagger}}\centerRV{X}}_{\psi_2}$ is the sub-Gaussian
norm of the standardization of $X$.
Provided $X$ is strictly sub-Gaussian for some $K > 0$, see \eqref{eqn:norm_var_proxy_new},
 we have $\textN{\sqrt{\VSigma^{\dagger}}\centerRV{X}}_{\psi_2}^2 \leq K
\| \textCovv{\sqrt{\VSigma^{\dagger}}\centerRV{X}} \|_2 = K$. It follows that for such distributions
$\textN{\sqrt{\VSigma^{\dagger}}\centerRV{X}}_{\psi_2}$ has a negligible effect on estimating precision matrices.

Second, similar to Lemma \ref{lem:directed_covariance_estimation}, the bound \eqref{eq:cross_term_result_arbitrary_A_B} depends only on components of $X$ induced by $\Amat$ and $\Bmat$. Since the sub-Gaussian norm can be interpreted as a proxy for the variance, this improves non-directional bounds
whenever the eigenvalues of $\Amat \VSigma^{\dagger}\Amat^\top$ and $\Bmat \VSigma^{\dagger}\Bmat^\top$ are
small compared to those of $\VSigma^{\dagger}$.

Third, as soon as $N > C(\rank(\VSigma) + u)\textN{\sqrt{\VSigma^{\dagger}}\centerRV{X}}_{\psi_2}^4$,
the estimation rate in \eqref{eq:cross_term_result_arbitrary_A_B} is similar to
the covariance estimation rate in Lemma \ref{lem:directed_covariance_estimation}.
That is, assume we are trying to estimate $\VSigma^{\dagger}$ with the sample covariance matrix of the random vector $Z = \VSigma^{\dagger}X$ through iid. copies $Z_i=\VSigma^{\dagger}X_i$.
In that case Lemma \ref{lem:directed_covariance_estimation} gives precisely the bound \eqref{eq:cross_term_result_arbitrary_A_B}.
This should come as a bit of a surprise, since it implies that estimating the precision matrix through the inverse of the sample covariance has the same theoretical guarantees as if we had access to a random vector $Z$  whose covariance is exactly $\VSigma^\dagger$.
To further stress this point, we now compare Theorem \ref{thm:condition_free_directional_precision_matrix_estimation_arbitrary}
{for $\Amat = \Bmat = \Id_D$ with the bound
\begin{equation}
\label{eq:general_perturbation_bound_that_sucks}
\textN{\hat \VSigma^{\dagger} - \VSigma^{\dagger}}_2 \lesssim \textN{\VSigma^\dagger}_2^2 \textN{\VSigma}_2\sqrt{\frac{D+u}{N}},
\end{equation}
which was derived in Section \ref{subsec:soa_precision_matrix} using general perturbation bounds for the matrix inverse.}
\begin{corollary}
\label{cor:condition_free_precision_matrix_estimation}
There exists a uniform constant $C>0$ such that if $N > C(\rank(\VSigma) + u)\textN{\sqrt{\VSigma^{\dagger}}\centerRV{X}}_{\psi_2}^4$
for any $u > 0$ with probability at least $1-\exp(-u)$ we have
\begin{align}
\label{eq:undirectional_result}
\textN{\hat \VSigma^{\dagger} - \VSigma^{\dagger}}_2 \lesssim \textN{\VSigma^{\dagger} \tilde X}_{\psi_2}^2 \sqrt{\frac{\rank(\VSigma)+u}{N}}.
\end{align}
\end{corollary}

\noindent
Assuming the squared sub-Gaussian norm is a good proxy for the variance, i.e. if
\eqref{eqn:norm_var_proxy_new} holds, the right hand side in \eqref{eq:undirectional_result}
becomes $ K \|{\VSigma^{\dagger}}\|_{2}\sqrt{(\rank(\VSigma)+u)/N}$. Compared to \eqref{eq:general_perturbation_bound_that_sucks},
this shows that using a general perturbation bound overestimates the influence the matrix condition
number $\N{\VSigma}_2\N{\VSigma^{\dagger}}_2$ has on precision matrix estimation.
The discrepancy between these two results suggests that the finite sample covariance estimator induces
a specific type of a perturbation that performs a form of regularization when estimating the inverse.

An immediate corollary is a relative bound for the smallest eigenvalue of $\VSigma$.
\begin{corollary}
Let  $d := \rank(\VSigma)$
and assume $X$ satisfies \eqref{eqn:norm_var_proxy_new} for some $K > 0$. There exists a constant $C>0$ such that provided
$N > C(d+u)\textN{\sqrt{\VSigma^{\dagger}}\centerRV{X}}_{\psi_2}^4$,
we have with probability at least $1-\exp(-u)$
\begin{align}
\label{eq:relative_eigenvalue}
\bigg\lvert{\frac{\lambda_d(\VSigma)}{\lambda_d(\hat \VSigma)} - 1}\bigg\rvert\lesssim K\sqrt{\frac{d+u}{N}}.
\end{align}
If \eqref{eqn:norm_var_proxy_new} does not hold the right hand side of \eqref{eq:relative_eigenvalue} is, up to absolute constants, replaced by $\lambda_d(\VSigma)\textN{\VSigma^{\dagger} \centerRV{X}}_{\psi_2}^2\sqrt{(d+u)/N}$.
\end{corollary}
\noindent
The bound \eqref{eq:relative_eigenvalue} is similar to \cite[Theorem 2.2]{yaskov2014lower}, which holds for isotropic
random variables with finite fourth order moments. Moreover, it is possible to avoid the dependence on $d$ if the number of samples $N$
is large compared to
$
\sum_{i=1}^{d-1}\frac{\lambda_i(\VSigma)}{\lambda_i(\VSigma) - \lambda_d(\VSigma)},
$
by using relative eigenvalue bounds from \cite[Theorem 2.15]{reiss2020nonasymptotic}.

\paragraph{Numerical validation.}
To validate the results of Theorem \ref{thm:condition_free_directional_precision_matrix_estimation_arbitrary} we consider
$X \sim \CN(\boldsymbol{0},\VSigma)$ for two types of covariance matrices $\VSigma \in \bbR^{10\times 10}$:
\begin{enumerate}[label=\textsf{Setting} \arabic*:,leftmargin = \widthof{[\textsf{Setting} 2:]}]
\item Set $\VSigma\!=\!\Umat\Smat\Umat^\top\!\!$, for $\Umat\!\in\!\bbR^{10\!\times\!10}$ sampled uniformly at random from the space of orthonormal
matrices, and $\Smat\!=\! \Diag{1,1,1,1,1,\nu,\nu,\nu,\nu,\nu}$. We consider $\nu = 10^{-j+1}$ for $j \in [10]$, which
implies that the matrix condition number $\textN{\VSigma}_2\textN{\VSigma^{\dagger}}_2$
ranges from $1$ to $10^{9}$.
\item Set $\VSigma_{i,j} = \nu^{\SN{i-j}}$, with $\nu \in \{0.5,0.55,\ldots,0.9,0.95\}$. This is a common model for distributions
where entries of $X$ correspond to values of a certain feature at different time stamps. It leads to correlated entries
when the time stamps are close by, i.e. when $\SN{i-j}$ is small.
\end{enumerate}
We choose matrices $\Amat$ and $\Bmat$ by sampling an orthoprojector
$\Amat$ uniformly at random with $\rank(\Amat) = 3$ and setting
$\Bmat = \Id_D - \Amat$. We repeat each experiment $100$ times and report the averaged relative errors in Figure \ref{fig:results_precision_matrices}.
Different lines of the same color correspond to different values of $\nu$.

\begin{figure*}[t!]
    \centering
    \subfigure[Setting 1]{\includegraphics[width=0.49\textwidth]{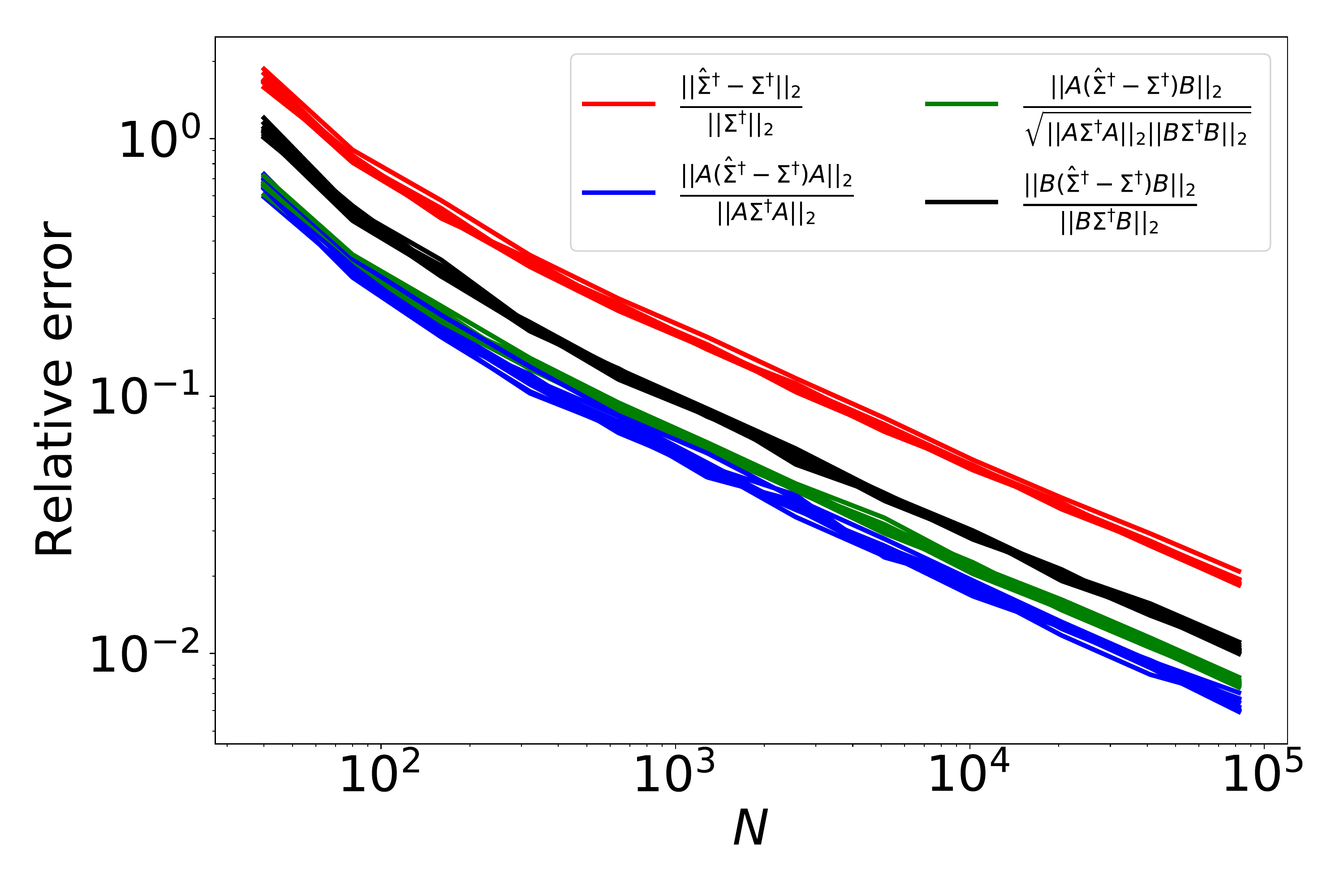}}
    \subfigure[Setting 2]{\includegraphics[width=0.49\textwidth]{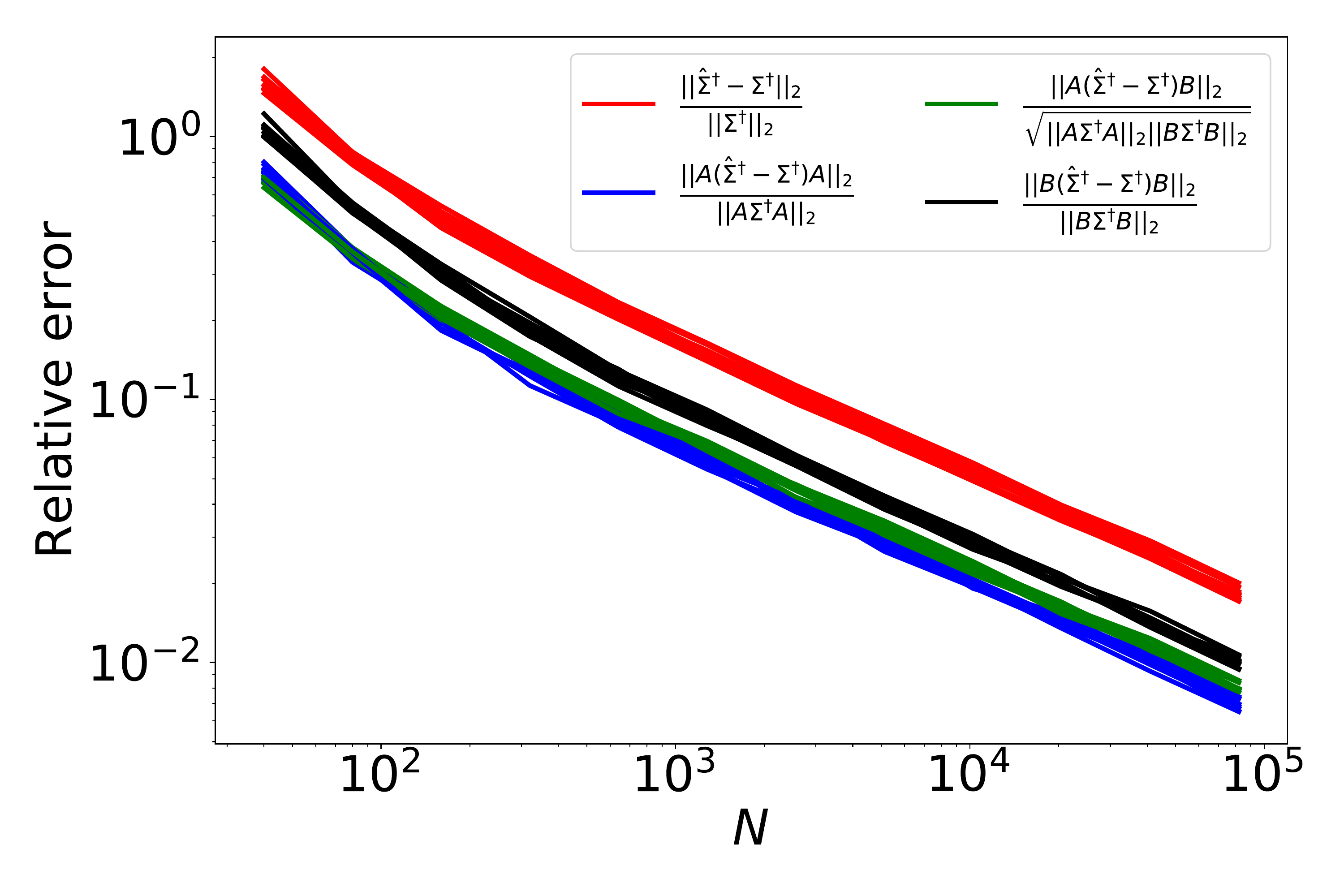}}
    \caption{Relative errors for directional and isotropic precision matrix estimation. Different colors correspond to different
    directions of evaluation (see legend). Per each color we plot $10$ lines corresponding to the $10$ parameter choices of $\nu$ in
    the construction of $\VSigma$. The plots show that the relative error does not depend on $\nu$, despite the fact that changing $\nu$ changes
    the sub-Gaussian condition number by a factor of $10$.}
    \label{fig:results_precision_matrices}
\end{figure*}

The estimation error shows the expected $N^{-1/2}$ rate.
Moreover, directional errors show  a clear dependence on the directional spectral norm.
On the other hand, since we are sampling Gaussians the squared sub-Gaussian norm is a proxy for the variance of the given random vector, and the results confirm that the error term in Theorem \ref{thm:condition_free_directional_precision_matrix_estimation_arbitrary} indeed scales with the corresponding (directional) sub-Gaussian norm.
Lastly,  although the condition number of $\VSigma$ depends on $\nu$, the specific value of $\nu$ does not affect how accurate $\hat \VSigma^{\dagger}$ is, as predicted by the theory.

\paragraph{Bounded random vectors.}
As mentioned in Remark \ref{rem:bounded}, a stronger form of direction dependent covariance and precision matrix estimation bounds hold for bounded random vectors.
The proofs for the bounded case follow along similar lines as for the sub-Gaussian case, except for the use of a slightly different probabilistic argument.
We now state only the results for the estimation of covariance and precision matrix, since the remaining bounds follow by analogy.

\begin{theorem}\label{thm:bounded_everything}
Let $\Amat \in \bbR^{d_1 \times D}, \Bmat\in \bbR^{d_2 \times D}$.
Assume  $\textN{\Amat\centerRV{X}}_2\leq C_A$, $\textN{\Bmat\centerRV{X}}_2\leq C_B$
almost surely.
Then for any $u>0$ with probability at least $1 - \exp(-u)$ we have
\begin{equation}
\label{eq:directed_covariance_estimation_bounded}
\textN{\Amat(\hat \VSigma - \VSigma) \Bmat^\top}_F \lesssim C_A C_B\sqrt{\frac{1+u}{N}}.
\end{equation}
Assume
$\textN{\Amat\VSigma^{\dagger} \centerRV{X}}_2\leq C_A^\dagger, \textN{\Bmat\VSigma^{\dagger} \centerRV{X}}_2\leq C_B^\dagger$,
$\textN{\sqrt{\VSigma^{\dagger}}\centerRV{X}}_2^2\leq \Theta $ almost surely.
There exists $C>0$ such that provided $N > C(1+u)\Theta^2$,  for any $u>0$ with probability at least $1-\exp(-u)$ we have
\begin{align}
\label{eq:cross_term_result_arbitrary_A_B_bounded}
\textN{\Amat({\hat \VSigma^{\dagger} - \VSigma^{\dagger}})\Bmat^\top}_F \lesssim C_A^\dagger C_B^\dagger\sqrt{\frac{1+u}{N}}.
\end{align}
\end{theorem}

\section{Application to single-index model estimation}
\label{sec:application}
In this section we use the results of Sections \ref{sec:covariance_estimation} and \ref{sec:precision_matrix_estimation} to establish concentration bounds
for estimating the index vector $\abs$ in the single-index model $\bbE[Y|X=x] = g(\abs^\top x)$.
Moreover, directional terms that arise from using directional dependent matrix concentration bounds provide an insight into how to further improve the performance of the standard estimator.
The second half of the section is thus devoted to describing this strategy (which is based on splitting up the data, conditioning and averaging), proving the error estimates, and providing numerical evidence to show and examine the claimed performance gains.

Throughout the section, we use a directional sub-Gaussian condition number
\begin{align*}
\kappa(\Pmat, X) := \textN{\Pmat\VSigma^{\dagger}\centerRV{X}}_{\psi_2}^2\textN{\Pmat\centerRV{X}}_{\psi_2}^2, \quad \Pmat \textrm{ is an orthoprojector}.
\end{align*}
This quantity can be seen as a restricted condition number $\textN{\Pmat\VSigma^{\dagger}\Pmat}_2 \textN{\Pmat \VSigma \Pmat}_2$
for the matrix $\VSigma = \Covv{X}$. Indeed, for strict sub-Gaussians, i.e. those satisfying \eqref{eqn:norm_var_proxy_new} for some $K$,
the two are equal up to a factor depending on $K$.

\subsection{Ordinary least squares for the single-index model}
\label{subsec:ordinary_SIM}
Single-index model (SIM) is a popular semi-parametric regression model that poses the relationship between the features $X \in \bbR^D$  and responses $Y \in \bbR$
as $\bbE[Y|X] = f(\abs^\top X)$, where $f$ is an unknown link function and $\abs \in \bbS^{D-1}$ is an unknown index vector.
SIM was developed in the 80s and 90s \cite{brillinger1982generalized,ichimura1993semiparametric} as an extension of generalized linear regression that does not specify the link function,
and which could thus avoid errors incurred by model misspecification.
Common applications are in econometrics \cite{dobson2008introduction,mcaleer2008single} and signal processing under
sparsity assumptions on the index vector \cite{plan2016high,plan2016generalized}.
It has been shown, e.g. in  \cite{gyorfi2006distribution}, that (in certain scenarios) the minimax estimation rate of SIM equals that of nonparametric univariate regression.

Methods for estimating the SIM from a finite data set $\{(X_i,Y_i) : i \in [N]\}$ often first construct an approximate index vector
$\hat \abs$, and then use nonparametric regression on $\{(\hat \abs^\top X_i, Y_i) : i \in [N]\}$ to estimate the link function.
With such an approach the generalization error of the resulting estimator depends
largely on the error incurred by estimating the index vector. Thus, the construction of $\hat \abs$ becomes the critical point.

An efficient approach, which first appeared in
\cite{li1989regression,brillinger1982generalized}, and later in modified forms in \cite{hristache2001direct,balabdaoui2019score},
is to solve the ordinary least squares (OLS) problem
\begin{align}
\label{eq:sim_estimator}
(\hat c, \hat \bbs) = \argmin_{c \in \bbR,\ \bbs \in \bbR^D}\sum_{i=1}^{N}\left(Y_i - c - \bbs^\top(X_i- \hat\Vmu_X)\right)^2, \text{ where } \hat\Vmu_X = \sum_{i=1}^N \frac{X_i}{N},
\end{align}
and then set $\hat \dbs := \hat \bbs/\| \hat \bbs\|_{2}$.
It was shown in \cite{balabdaoui2019score} that $\sqrt{N}(\hat \dbs - \abs)$ is  asymptotically
normal with mean zero, provided $X$ has an elliptical distribution, $f$ is a non-decreasing function that is strictly increasing
on some non-empty sub-interval of the support of $\abs^\top X$.
Assuming only ellipticity of $X$ and statistical independence
of $Y$ and $X$ given $\abs^\top X$, the population vector $\bbs:=\VSigma^{\dagger}\Covv{X,Y}$
is still contained in $\Span\{\abs\}$, see e.g. \cite[Proposition 3]{klock2020estimating}.
Provided $\bbs \neq 0$, in these cases the direction $\dbs := \bbs/\N{\bbs}_2$ equals the index vector $\abs$ up to sign.
Thus, $\hat \dbs$ is in many cases a consistent estimator of the index vector $\abs$, and under certain conditions it converges with an $N^{-1/2}$ rate.

The minimal $\textN{\cdot}_2$-norm solution of \eqref{eq:sim_estimator} admits a closed form
\begin{align}
\label{eq:finite_sample_sim_estimator}
\hat c = \hat \Vmu_Y,\quad \hat \bbs = \hat \VSigma^{\dagger}\hat \rbs,\textrm{ with }
\hat \Vmu_Y \!:=\! \!\sum_{i=1}^{N} \frac{Y_i}{N},\quad \hat \rbs\!:=\!\!\sum_{i=1}^{N}\frac{(X_i - \hat \Vmu_X)(Y_i - \hat \Vmu_Y)}{N}.
\end{align}
Using the results of the previous two sections we can show a direction
dependent concentration bound for the vector $\hat \bbs$.
\begin{lemma}
\label{lem:directional_linear_regression}
Let $Y \in \bbR$ be sub-Gaussian. Denote  $\Pmat=\dbs\dbs^\top$, $\Qmat := \Id_D - \Pmat$, and $\kappa_{\Pmat\Qmat} = \kappa(\Pmat,X) \vee \kappa(\Qmat,X)$.
There exists $C>0$ such that provided $N > C(\rank(\VSigma)+u)\textN{\sqrt{\VSigma^{\dagger}}\centerRV{X}}_{\psi_2}^4$,
for any $u>0$, with probability at least $1 - \exp(-u)$ we have
\begin{align}
\label{eq:lin_reg_bounds_dir_P}
&\textN{\Pmat\big(\bbs - \hat \bbs\big)}_2 \lesssim\textN{\centerRV{Y}}_{\psi_2}\textN{\Pmat\VSigma^{\dagger}\centerRV{X}}_{\psi_2}\sqrt{\kappa_{\Pmat\Qmat}}\sqrt{\frac{\rank(\VSigma)+u}{N}},\\
\label{eq:lin_reg_bounds_dir_Q}
&\textN{\Qmat\big(\bbs - \hat \bbs\big)}_2 \lesssim\textN{\centerRV{Y}}_{\psi_2}\textN{\Qmat\VSigma^{\dagger}\centerRV{X}}_{\psi_2}\sqrt{\kappa_{\Pmat\Qmat}}\sqrt{\frac{\rank(\VSigma)+u}{N}}.
\end{align}
\end{lemma}
\noindent
For the normalized vector, respectively the directions, the following bound holds.
\begin{corollary}
\label{cor:directional_normalized_linear_regression}
Assume the setting of Lemma \ref{lem:directional_linear_regression}.
There exists a constant $C>0$ such that
for any $u>0$, with probability at least $1-\exp(-u)$  we have
\begin{align}
\label{eq:normed_lin_reg_bounds_dir}
&\N{\hat \dbs - \dbs}_2 \lesssim\frac{\textN{\centerRV{Y}}_{\psi_2}\textN{\Qmat\VSigma^{\dagger}\centerRV{X}}_{\psi_2}\sqrt{
\kappa_{\Pmat\Qmat}}}{\N{\bbs}_2}\sqrt{\frac{\rank(\VSigma)+u}{N}},\quad \textrm{provided that }\\
\label{eq:normed_lin_reg_bounds_dir_N_requirement}
&N > C(\rank(\VSigma)+u)\bigg(\textN{\sqrt{\VSigma^{\dagger}}\centerRV{X}}_{\psi_2}^4 \vee \frac{\textN{\centerRV{Y}}_{\psi_2}^2\textN{\Pmat\VSigma^{\dagger}\centerRV{X}}_{\psi_2}^2\kappa_{\Pmat\Qmat}}{\N{\bbs}^2_2}\bigg).
\end{align}
\end{corollary}

\noindent
As mentioned before, under certain conditions we have $\dbs = \abs$, where $\abs$ is the index vector in the given SIM.
In such cases Corollary \ref{cor:directional_normalized_linear_regression} confirms that $\hat\dbs$ is a consistent estimator of $\abs$ and achieves a $N^{-1/2}$ convergence rate, which has been observed in previous works.
To provide some context for our result we give a brief
description of the two most popular strategies for estimating $\abs$.

The first group of methods can be distinguished by their simplicity and efficiency
as they estimate the index vector using empirical estimates of
first and second order moments of random variables $X$ and $Y$. The studied OLS estimator (sometimes
also referred to as the average derivative estimator \cite{brillinger1982generalized,li1989regression})
belongs to this group. Inverse regression based techniques also
fall into this category, see e.g. \cite{li1991sliced,dennis2000save,li2007directional} or the review \cite{ma2013review},
and their convergence rate usually equals $N^{-1/2}$ and is thus on par with OLS.
A drawback caused essentially by the simplicity of these methods is that theoretical guarantees typically require $X$ to be an elliptical distribution, and
sometimes a form of monotonicity is needed to avoid issues that arise when dealing with symmetric functions such as $\xbs\mapsto (\abs^\top \xbs)^2$.

The second, more sophisticated but also computationally heavier, class of methods is based on
solving more complicated optimal programs for the recovery of $\abs$, to which a closed form solution does not exist. This includes non-parametric methods for sufficient dimension reduction, e.g. \cite{dalalyan2008new}, aiming
at estimating the gradients of $f(\abs^\top X)$ at all training samples $X_i$; methods based on combining
index estimation and monotonic regression, e.g. \cite{kalai2009isotron,kakade2011efficient};
and methods that simultaneously estimate the index vector and use spline interpolation \cite{radchenko2015high,kuchibhotla2020efficient}, to name just a few.
For these methods convergence rates up to $N^{-1/2}$ (sometimes slower) can typically be proven, and
the theory requires less stringent assumptions than those exhibited by the first class of methods.

A new insight in Corollary \ref{cor:directional_normalized_linear_regression}, which has
not been emphasized before, is the directionally-dependent influence of the spectrum of $\VSigma$ on the estimation guarantee.
To make this more precise, we now consider a special case when $X$ is strictly sub-Gaussian,
the link function $f$ is Lipschitz-smooth, and the index $\abs$ is an eigenvector of $\VSigma$. We note that similar results hold
if $\abs$ is an approximate eigenvector of $\VSigma$, in the sense that
$\abs^\top \VSigma^{\dagger}\abs  \approx (\abs^\top \VSigma \abs)^{-1}$.

\begin{corollary}
\label{cor:motivation_slicing_and_averaging}
Let $Y$ be sub-Gaussian, $X$ be strictly sub-Gaussian (i.e. satisfying \eqref{eqn:norm_var_proxy_new} for some $K > 0$)
and assume $Y = f(\abs^\top X) + \zeta$ for $\bbE\zeta = 0$ with $\sigma_{\zeta}:=\textN{\zeta}_{\psi_2}<\infty$.
Assume consistent estimation, i.e. $\dbs = \abs$, and that $\VSigma \abs = \sigma_{\Pmat}^2 \abs$.
Then there exists $C_K>0$, depending only on $K$, so that if
$$N > C_K(\rank(\VSigma)+u)\frac{(L+\sigma_{\zeta}\sigma_{\Pmat}^{-1})^2\kappa(\Qmat,X)}{\N{\bbs}_2^{2}},$$
for any $u>0$ we have with probability at least $1-\exp(-u)$
\begin{align}
\label{eq:normed_lin_reg_bounds_dir_modified}
&\N{\hat \dbs - \abs}_2 \lesssim \frac{\left(L\sigma_{\Pmat} + \sigma_{\zeta}\right)\sqrt{\N{\Qmat\VSigma^{\dagger}\Qmat}_2\kappa(\Qmat,X)}}{\N{\bbs}_2}\sqrt{\frac{\rank(\VSigma)+u}{N}}.
\end{align}
\end{corollary}

Corollary \ref{cor:motivation_slicing_and_averaging} shows that the variance in the direction of the index vector,
$\sigma_\Pmat^2 = \textVar{\abs^\top X}$, and the variance in orthogonal directions, quantified through the spectrum of $\Qmat\VSigma^{\dagger} \Qmat$, influence index vector estimation in a significantly different manner.
Namely, as long as $\sigma_\Pmat \gg \sigma_{\zeta}$,
smaller $\sigma_\Pmat$ has a provably beneficial effect on estimation accuracy, whereas small
non-zero eigenvalues of $\VSigma$ (i.e. large eigenvalues of $\VSigma^\dagger$), corresponding to eigenvectors in $\Im(\Qmat)$, can only worsen the accuracy.
A similar observation can be made when inspecting the asymptotic covariance
of $\sqrt{N}(\hat \dbs - \abs)$, which after using $\VSigma \abs = \sigma_\Pmat^2 \abs$ and \cite[Theorem 1]{balabdaoui2019score}, is given by
$$
\underbrace{\frac{\sigma_\Pmat^2}{\Covv{f(\abs^\top X),\abs^\top X}}}_{=:T_1}\underbrace{\Qmat \VSigma^{\dagger}\Qmat\Covv{(\tilde Y - \tilde X^\top \bbs)\tilde X} \Qmat\VSigma^{\dagger}\Qmat}_{=:T_2}.
$$
The scalar factor $T_1$ is comparable to $\N{\bbs}_2^2$,
and the matrix $T_2$ shows that the variance of the estimator grows with large eigenvalues
of $\Qmat \VSigma^{\dagger}\Qmat$.
On the other hand, the variance of the residual $\tilde Y - \tilde X^\top \bbs$ decreases
if the accuracy of the linear fit of $\tilde X^\top \bbs \approx f(\abs^\top X)-\bbE f(\abs^\top X)$ improves,
which can typically be observed under monotonic links $f$ and decreasing $\sigma_\Pmat^2 = \textVar{\abs^\top X}$.

We will now use this observation as a guiding principle for developing a modified estimator
that splits the data into subsets with small $\sigma_\Pmat$'s, computes corresponding OLS vectors on each subset,
and then combines local estimators into a global estimator by weighted averaging.

\subsection{Averaged conditional least squares for the single-index model}
\label{subsec:conditional_SIM}
\begin{figure*}[t!]
    \centering
    \includegraphics[width=0.7\textwidth]{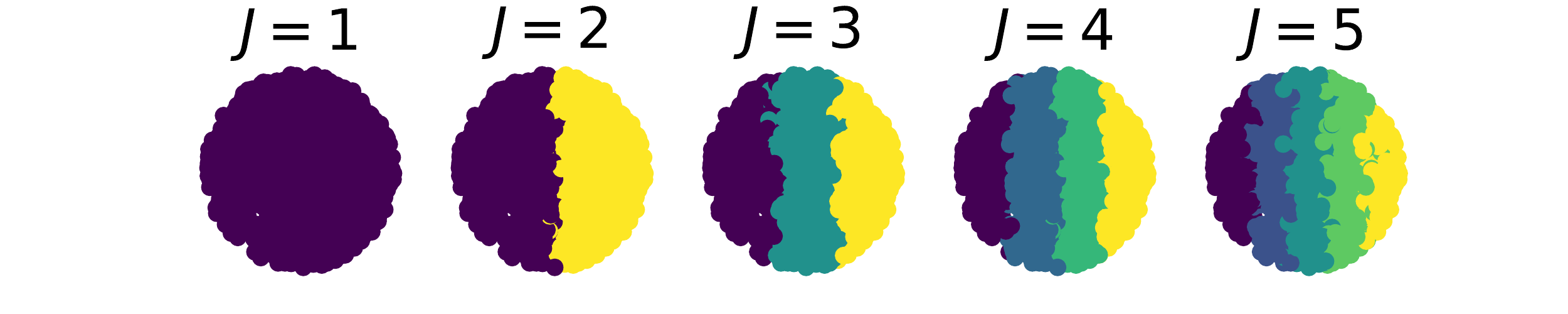}
    \caption{$(X,Y)$ sampled according to $X \sim \textrm{Uni}(\{X : \N{X}_2\leq 1\})$ and $Y = f(\abs^\top X) + \zeta$,
    where $\Span{\{\abs\}}$ is the horizontal line, and $\zeta \sim \CN({0},0.01\Var{f(\abs^\top X)})$. We use a dyadic level-set partitioning of the data into $J$ sub-intervals. The color indicates the labeling.}
    \label{fig:levelset_details_sigmoid}
\end{figure*}

We now study the just discussed alternative procedure, where we first split the data into
subsets, aiming to reduce the variance of the data distribution in the direction of the index vector, and then compute and average out the estimators
from each subset. Since we have no  a priori knowledge about the index vector, constructing such a partition seems challenging.
However, when the link function is monotonic the partitioning is induced by a decomposition
of $\Im(Y)$ into equisized intervals, see Figure \ref{fig:levelset_details_sigmoid}.
For the sake of simplicity, in the following we assume $Y \in [0,1)$ holds almost surely.

For $\ell \in[J]$ let $\CR_{J,\ell} := [\frac{\ell-1}{J},\frac{\ell}{J})$ denote equisized regions partitioning $[0,1)$.
Furthermore, define herein called level-sets $\CS_{J,\ell} = \{(X_i, Y_i) : Y_i \in \CR_{J,\ell}\}$,
which induce a partition of the data set into $J$ subsets based on the responses. Then estimate
$\abs$ according to the following algorithm.
\begin{enumerate}[label=\textsf{Step }\arabic*,leftmargin=\widthof{[\textsf{Step}2)]}]
\item\label{step:1} Solve \eqref{eq:finite_sample_sim_estimator} on each subset $\CS_{J,\ell}$, by computing $\hat c_{J,\ell}\in\bbR$ and $\hat \bbs_{J,\ell}\in\bbR^D$.
\item\label{step:2} Define the empirical density $\hat \rho_{J,\ell} := \SN{\CS_{J,\ell}}/N$, set the thresholding parameter $\alpha > 0$, and compute the averaged outer product matrix
\begin{align}
\label{eq:def_MJ_SIM}
\hat \Mmat_J = \sum_{\ell = 1}^{J} \mathbbm{1}_{[\alpha J^{-1}, 1]}(\hat \rho_{J,\ell})\hat\rho_{J,\ell}\hat \bbs_{J,\ell}\hat \bbs_{J,\ell}^\top.
\end{align}
\item\label{step:3} Use the eigenvector corresponding to the largest eigenvalue of $\Mmat_J$, denoted as $\ubs_1(\hat \Mmat_J)$, as an approximation of the index vector $\abs$.
\end{enumerate}

The parameter $\alpha$ is used to promote numerical stability by suppressing the contributions of sparsely populated subsets.
In other words, we only keep those level-sets whose empirical mass behaves as if $Y$ were uniformly distributed over $\Im(Y)$ (which can in some problems be achieved by a suitable transformation of the responses).
The parameter $J$ on the other hand defines the number of sets we use in the partition of the given data set,
and dictates the trade-off between $\Var{\abs^\top X|Y \in \CR_{J,\ell}}$
and the number of samples $\lvert \CS_{J,\ell}\rvert$ in a given level set.

For a random vector $Z$ we denote a conditional random vector $Z_{J,\ell}:=Z|Y\in\CR_{J,\ell}$.
Note that $Z_{J,\ell}$ inherits sub-Gaussianity of $Z$ provided
$\bbP(Y \in \CR_{J,\ell}) > 0$, see Lemma \ref{lem:subG_conditional}.
We now  analyze the approach under the following assumption:
\begin{enumerate}[label=(\Alph*),leftmargin=\widthof{[(A)]}]
\item\label{ass:alignedness} $\bbs_{J,\ell}\!:=\! \Covv{X|Y \in \CR_{J,\ell}}^\dagger\Covv{X,Y|Y \in \CR_{J,\ell}}\in \Span\{\abs\}$ for all $\ell \in [J]$.
\end{enumerate}
Assumption \ref{ass:alignedness} is not particularly restrictive.
For example, it can be shown that \ref{ass:alignedness}
holds if $X$ is elliptically symmetric, which is a standard assumption when using the OLS
functional \eqref{eq:finite_sample_sim_estimator}, and if the function noise $Y-\bbE[Y|X]$ is statistically independent of $X$ given $\abs^\top X$ \cite[Proposition 3]{klock2020estimating}.

\begin{theorem}
\label{thm:conditional_sim_estimator}
Assume \ref{ass:alignedness} holds and that $Y \in [0,1)$ almost surely. Let $J > 0$, $\alpha > 0$, and assume we
are given $N$ iid. copies of $(X,Y)$. Denote $\Pmat := \abs\abs^\top$, $\Qmat := \Id_D - \Pmat$, $\VSigma_{J,\ell} := \Covv{X|Y \in \CR_{J,\ell}}$,
and
\[
\kappa_{J,\ell}\!:=\!\kappa(\Pmat,X_{J,\ell}) \vee \kappa(\Qmat, X_{J,\ell}) ,\textrm{ and } K_{J,\ell} :=\! \big\|{\sqrt{\VSigma_{J,\ell}^{\dagger}}\centerRV{X}_{J,\ell}}\big\|_{\psi_2}\!.
\]
Let $I_J:= \{\ell \in [J] : \hat \rho_{J,\ell} >\alpha J^{-1}\}$ be the index set containing the indices of active level-sets
and set $K_J := \max_{\ell \in I_J}K_{J,\ell}$.
There exists $C > 0$ such that, if
\begin{equation}
\label{eq:N_requirement_last}
N > CK_J^4\frac{J(\rank(\VSigma)+\log(J)+u)}{\alpha},
\end{equation}
there exists a sign $s \in \{-1,1\}$ so that for any $u>0$, with probability at least $1-\exp(-u)$ we have
\begin{align}
\label{eq:conditional_sim_equation}
\textN{s \ubs_1(\hat \Mmat_J)\! -\! \abs}_{2}^2 \!\leq \! {\varepsilon_{N,J,u}}\frac{\sum_{\ell \in I_J}\hat \rho_{J,\ell}\kappa_{J,\ell} \textN{\Qmat\VSigma_{J,\ell}^{\dagger}\centerRV{X}_{J,\ell}}_{\psi_2}^2}{\sum_{\ell\in I_{J}}\!\hat \rho_{J,\ell}\big(\textN{b_{J,\ell}}_2^2\! -\!{\varepsilon_{N,J,u}}\kappa_{J,\ell}\textN{\Pmat\VSigma_{J,\ell}^{\dagger}\centerRV{X}_{J,\ell}}_{\psi_2}^2\big)},
\end{align}
where \[\varepsilon_{N,J,u} := C\frac{\rank(\VSigma) + \log(J) + u}{\alpha J N}.\]
The same guarantee holds when replacing the denominator in \eqref{eq:conditional_sim_equation} with $\lambda_1(\hat \Mmat_J)$.
\end{theorem}
\noindent
The error rate in Theorem \ref{thm:conditional_sim_estimator} is dictated by $\varepsilon_{N,J,u}$
and there are two ways to interpret Theorem \ref{thm:conditional_sim_estimator}.
First, for a fixed parameter $J$ all terms in \eqref{eq:N_requirement_last} and \eqref{eq:conditional_sim_equation} (except $N$ and $\varepsilon_{N,J,u}$) are constants, and we obtain a  $N^{-1/2}$ convergence rate
for the non-squared error as soon as the sample size is sufficiently large. This is
the same rate as the one achieved with the standard OLS estimator \eqref{eq:finite_sample_sim_estimator}.

Parameter $J$ however provides additional flexibility and an intriguing option is to select $J$ as a function that grows with $N$ while still complying with \eqref{eq:N_requirement_last}.
Namely, using $u = \log(J)$ and $J\asymp N/\log(N)$, and plugging into \eqref{eq:conditional_sim_equation}, implies that
a $\sqrt{\log(N/\log(N))\log(N)}N^{-1}\leq \log(N)N^{-1}$ rate is achievable, so long as the remaining terms involving parameter $J$ are balanced.

To give an illustrative example of such a case, we consider the following idealized
setting. We assume the noise-free regime $Y = f(\abs^\top X)$ with strictly monotonic
link in the sense that
\begin{align}
\label{eq:monotonicity_assumption}
  \Covv{\abs^\top X, f(\abs^\top X)| Y \in \CR_{J,\ell}} > L \textrm{Var}(\abs^\top X|Y\in \CR_{J,\ell})\quad \textrm{for some }  L > 0.
\end{align}
Furthermore, conditional random variables
$X_{J,\ell}$ are assumed strictly sub-Gaussian, where the strict sub-Gaussianity constant
$K > 0$ in \eqref{eqn:norm_var_proxy_new} is independent of $J$ and $\ell$,
and we model local covariance matrices as
\begin{align}
\label{eq:idealized_setting}
\Covv{\tilde X_{J,\ell}} = C_{\ell} J^{-2}\Pmat + \Qmat\quad \textrm{ with } c_1 \leq C_{\ell} \leq c_2,
\end{align}
where $c_1,c_2$ are universal constants independent of $J$.
Condition \eqref{eq:idealized_setting} implies that partitioning the data into $J$ level-sets results in a reduction of the variance along the direction of the index
vector, while not affecting the variance in directions orthogonal to the index vector.

With strict sub-Gaussianity
and \eqref{eq:idealized_setting}, we have $K_J \leq \sqrt{K}$, $\kappa_{J,\ell} \leq (c_2/c_1) K^2$,
$\textN{\Pmat\VSigma_{J,\ell}^{\dagger}\widetilde X_{J,\ell}}_{\psi_2}^2 \leq KJ^2/c_1$,
and $\textN{\Qmat\VSigma_{J,\ell}^{\dagger}\widetilde X_{J,\ell}}_{\psi_2}^2 \leq K$.
Then, using $u = \log(J)$ and
$J = \tau N/\log(N)$ for some $\tau > 0$,
the result \eqref{eq:conditional_sim_equation} can be written as
\begin{align}
\label{eq:conditional_sim_equation_derived}
\textN{s \ubs_1(\hat \Mmat_J) - \abs}_{2}^2 \leq \! \frac{C_K \rank(\VSigma)}{\alpha \tau \sum_{\ell\in I_{J}}\hat \rho_{J,\ell}\left(\textN{b_{J,\ell}}_2^2 - \tau \frac{C_K(\rank(\VSigma)+1)}{\alpha}\right)}\!\frac{\log^2(N)}{N^2},
\end{align}
where $C_K > 0$ depends only on $K$, $c_1$, and $c_2$. Furthermore, using
Assumption \ref{ass:alignedness}, and \eqref{eq:idealized_setting}, we have
$\Covv{X,Y|Y \in \CR_{J,\ell}} = \Covv{a^\top X, Y | Y \in \CR_{J,\ell}}a$,
and with the asserted monotonicity \eqref{eq:monotonicity_assumption} we get
\begin{align}
\nonumber
\N{b_{J,\ell}}_2 &\geq \frac{\Covv{X,Y|Y \in \CR_{J,\ell}}^\top \VSigma^{\dagger}_{J,\ell}\Covv{X,Y|Y \in \CR_{J,\ell}}}{\N{\Covv{X,Y|Y \in \CR_{J,\ell}}}_2}\\
\label{eq:lower_bound_bJL}
&> L\textrm{Var}(a^\top X|Y\in \CR_{J,\ell}) a^\top  \VSigma^{\dagger}_{J,\ell} a = L.
\end{align}
Plugging this bound into \eqref{eq:conditional_sim_equation_derived}, and choosing $\tau \in (0,\alpha L^2/(2C_K(\rank(\VSigma)+ 1)))$,
we thus obtain a $\log^2(N)/N^{2}$ convergence rate for the squared error.

The described idealized setting can be a reasonable approximation of reality
for strictly monotonic link functions $f$ in the noise-free regime, since
\eqref{eq:idealized_setting} can be motivated by the observation that in this case slicing the distribution of $X$ based on $Y$
reduces the variance of the data along the direction of the index vector,
as shown in Figure \ref{fig:levelset_details_sigmoid}. Thus, monotonicity is a key requirement
for establishing a faster rate. In the case of noisy responses and strictly monotone links,
the matter is more delicate. Namely, by the minimax rate of linear regression, which equals $\textrm{Var}(Y-f(\abs^\top X))/N$, result \eqref{eq:conditional_sim_equation_derived}
can only hold for all $N \in \bbN$, if $\textrm{Var}(Y-f(\abs^\top X)) \in \CO(\log^2(N)/N)$
as $N \rightarrow \infty$.

The same observation carries over to
more general strictly monotonic links, because, in the presence of noise, condition \eqref{eq:idealized_setting}
and the lower bound \eqref{eq:lower_bound_bJL} do not hold for all $J \in \bbN$. For instance, if $J^2$ exceeds $1/\textrm{Var}(Y-f(\abs^\top X))$
or becomes even larger, the conditional distribution $Y|Y \in \CR_{J,\ell}$ has roughly as much correlation with label noise as it does with the labels, $f(\abs^\top X)$.
This suggests $\N{b_{J,\ell}}_2 \approx 0$, and that further slicing the distribution of $X$ based on $Y$
will not decrease the variance in index vector direction, as required per \eqref{eq:idealized_setting}.
For $J^{-2} \gg \textrm{Var}(Y-f(\abs^\top X))$ however, Condition \eqref{eq:idealized_setting} and
\eqref{eq:lower_bound_bJL} can still hold approximately, and our experiments below suggest that
the proposed estimator benefits from an aggresive splitting of the data
in the case of nonlinear, strictly monotonic links.
We now confer the implications of Theorem \ref{thm:conditional_sim_estimator},
with synthetic examples, under a data-driven parameter choice rule for $J$.

\begin{figure*}[t!]
    \centering
    \subfigure[Identity: $f(t) =t$]{\label{fig:sim_estimator_identity}\includegraphics[width=0.49\textwidth]{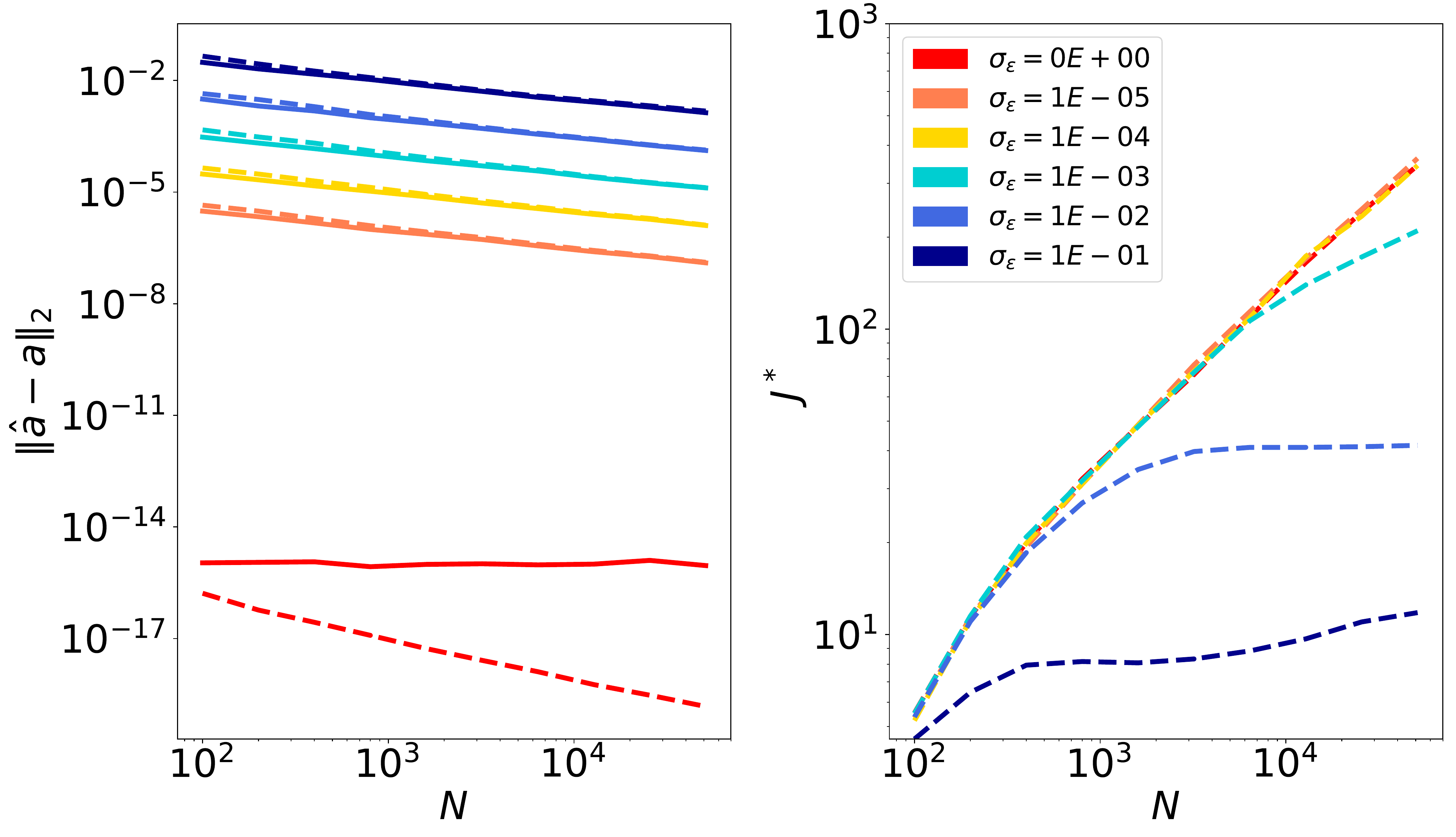}}
    \subfigure[Logit: $f(t) = \frac{1}{1+\exp(-t)}$]{\label{fig:sim_estimator_logit}\includegraphics[width=0.49\textwidth]{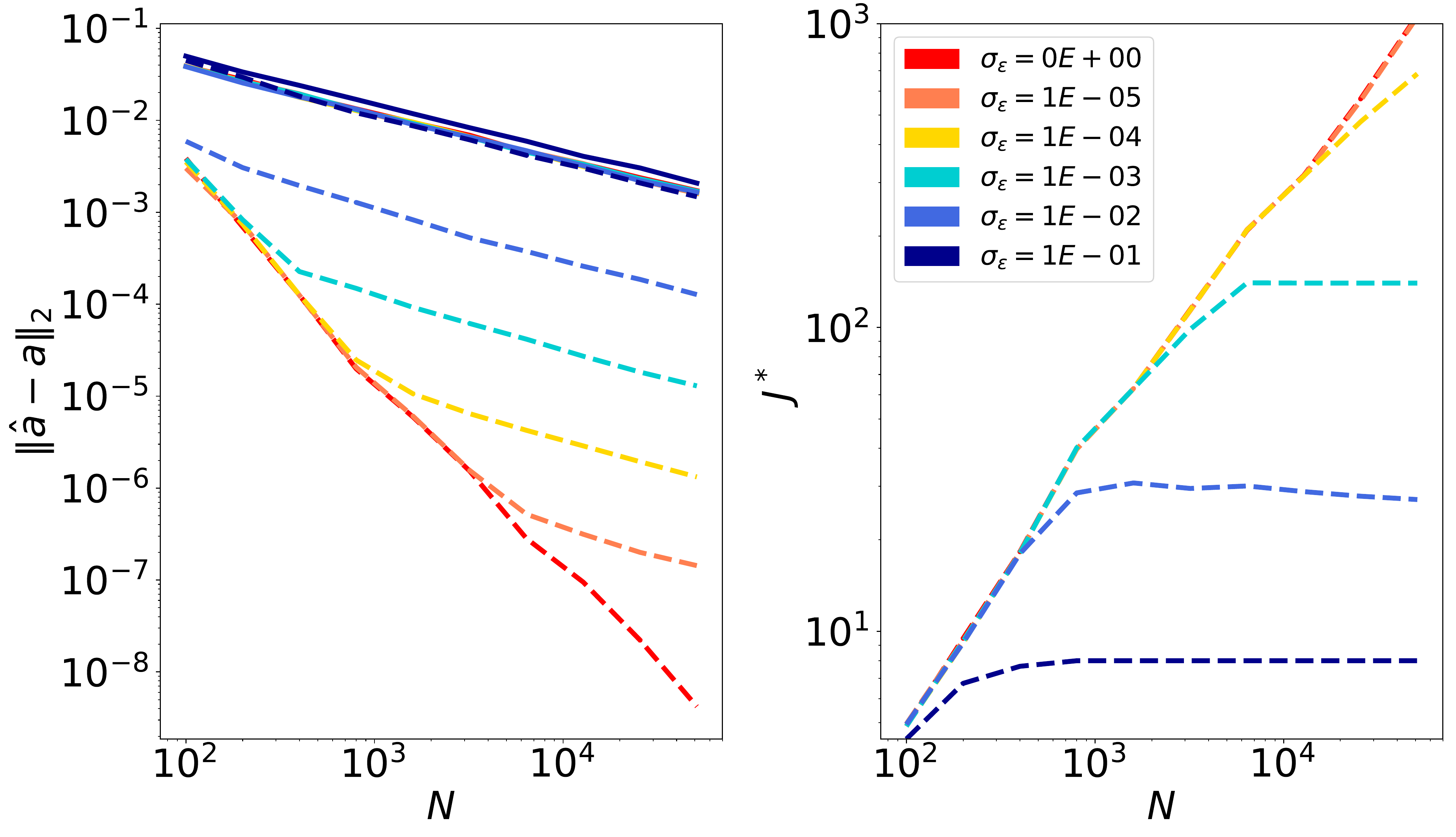}}
    \subfigure[ReLU: $f(t) = \max\{0,t\}$]{\label{fig:sim_estimator_ReLU}\includegraphics[width=0.49\textwidth]{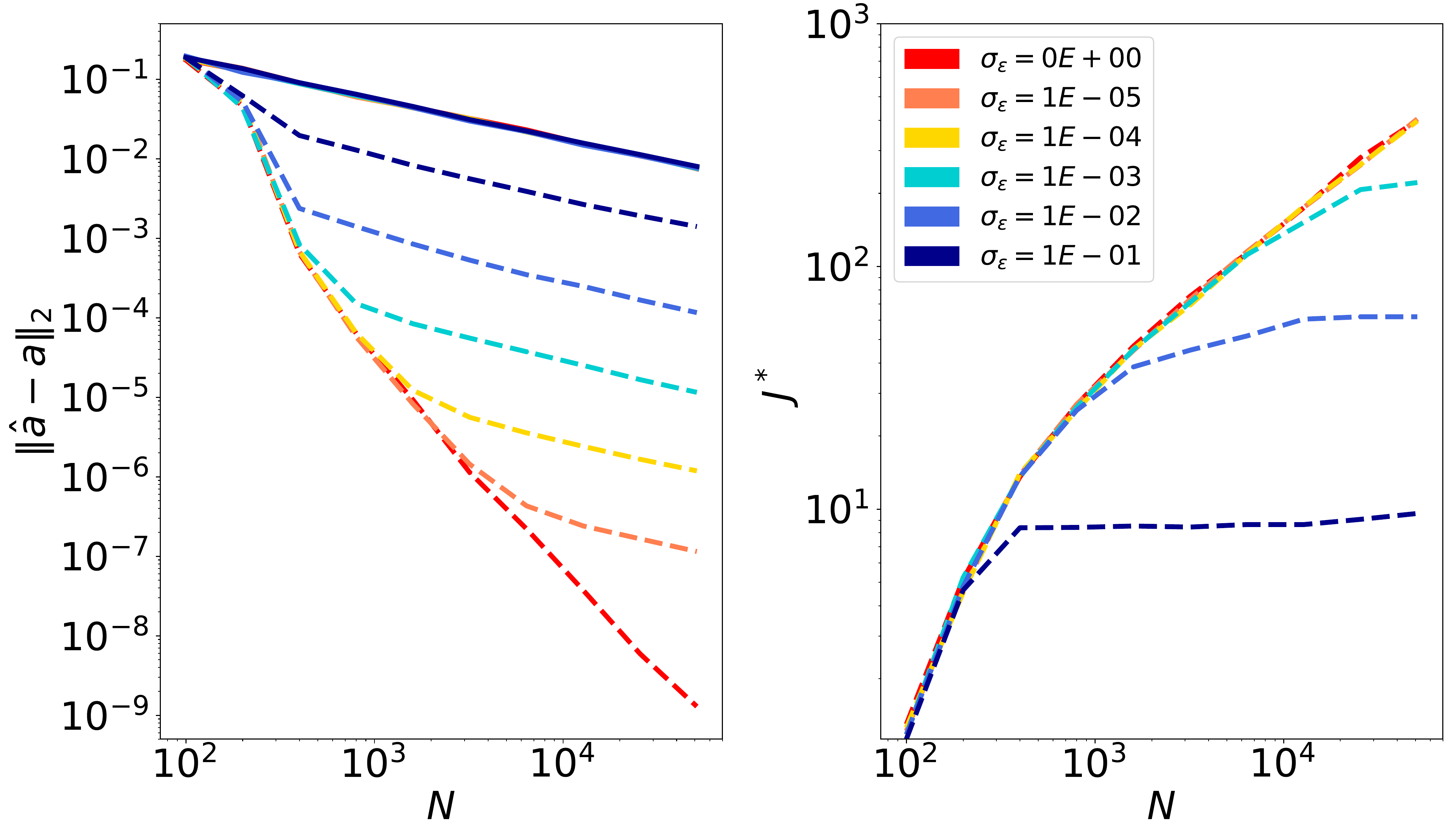}}
    \subfigure[Tangent hyperbolicus: $f(t) = \tanh(t)$]{\label{fig:sim_estimator_tanh}\includegraphics[width=0.49\textwidth]{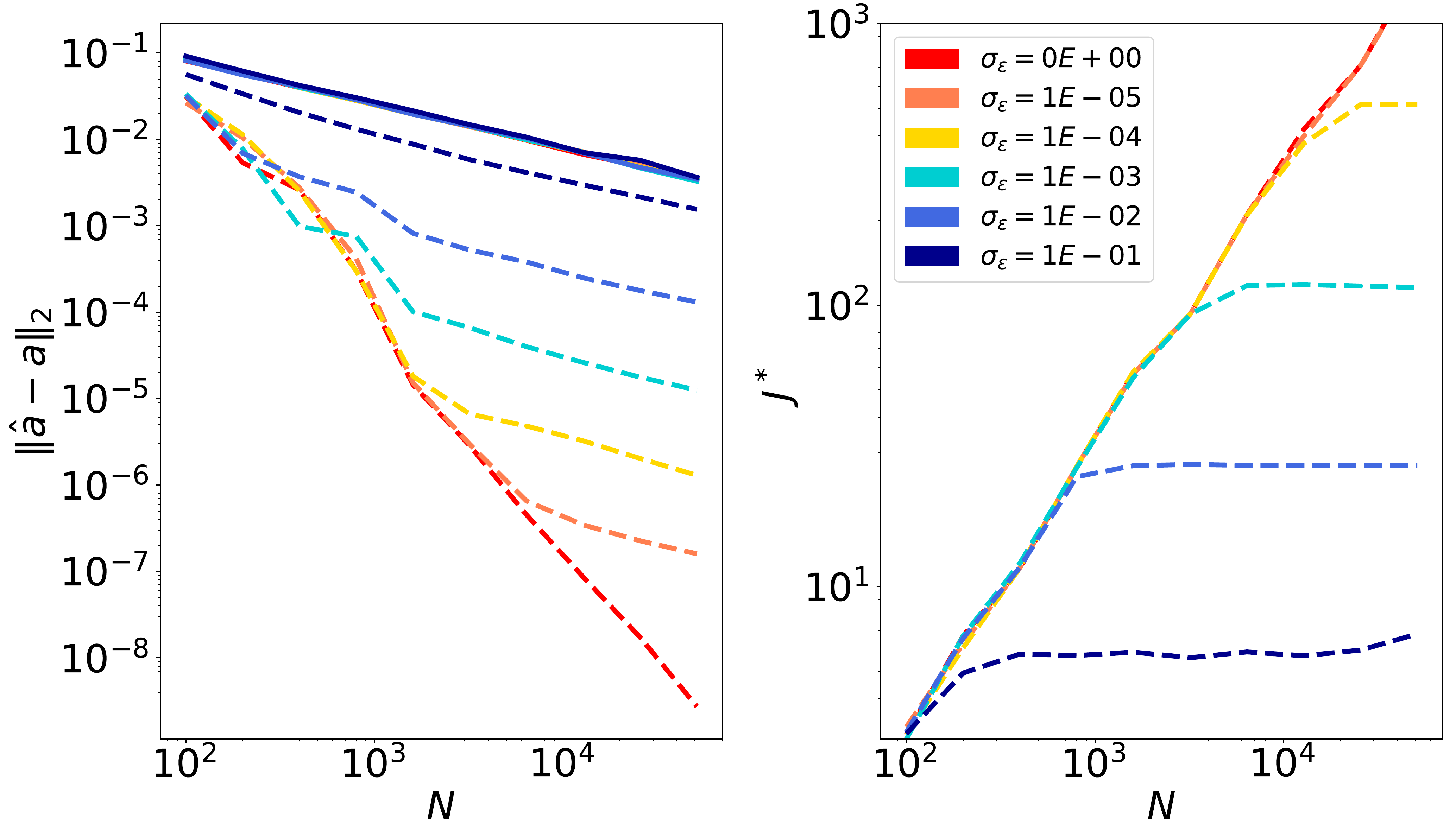}}
    \subfigure[Shifted absolute: $f(t) = \SN{t-\frac{1}{2}}$]{\label{fig:absolute_1}\includegraphics[width=0.49\textwidth]{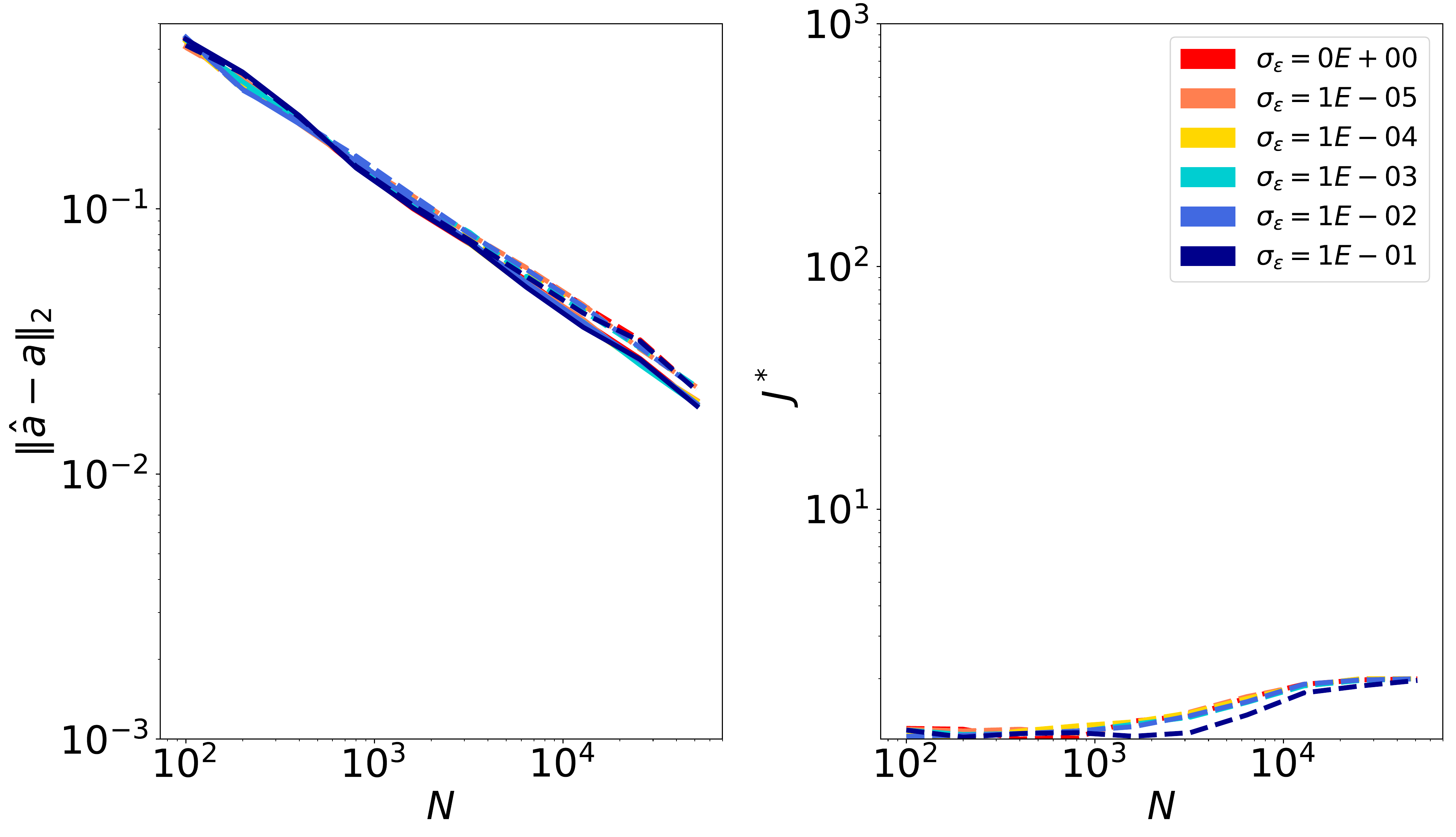}}
    \subfigure[Mixed: $f(t) = t+t^2+\cos(t)$]{\label{fig:weird_function_1}\includegraphics[width=0.49\textwidth]{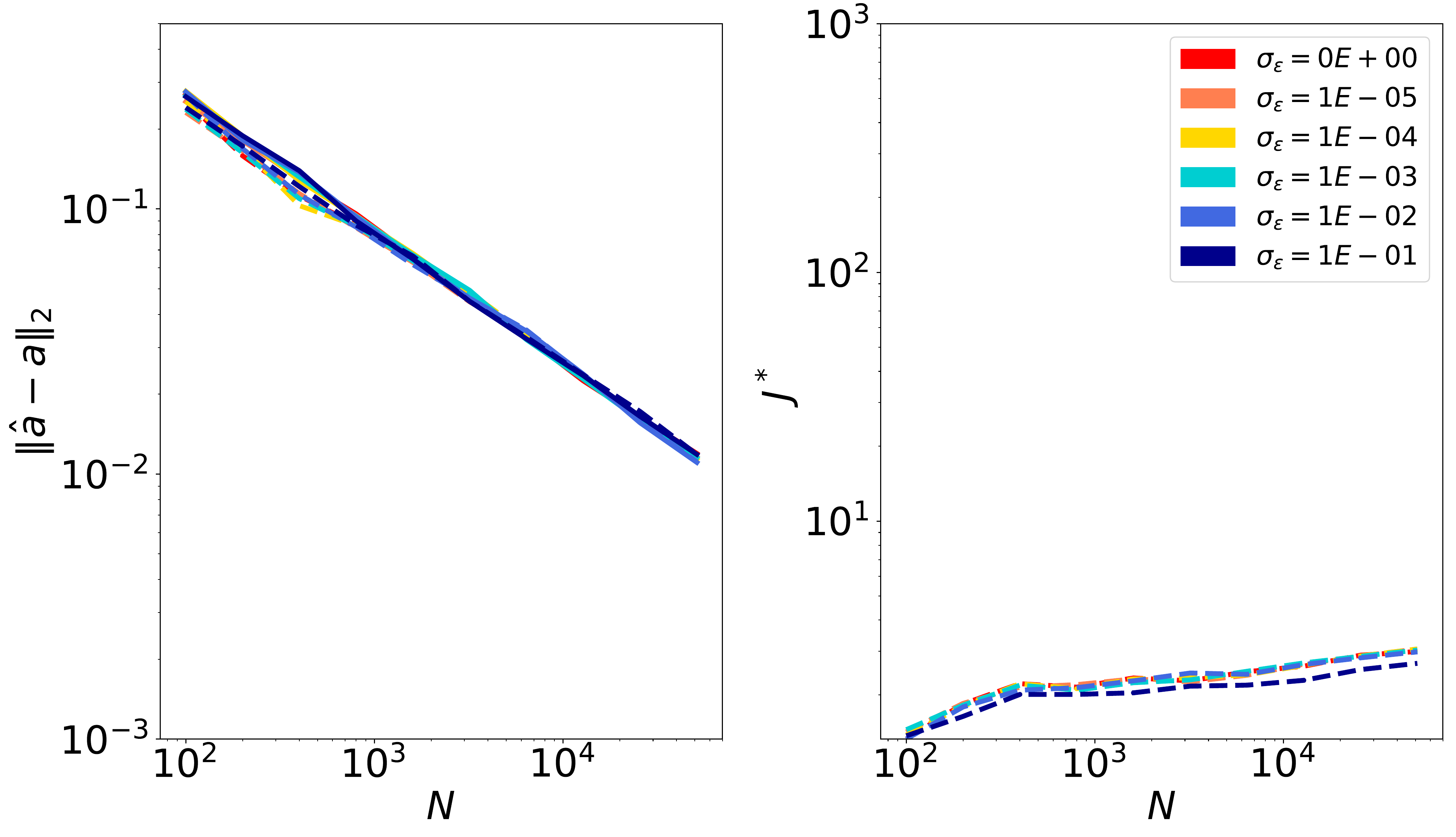}}
    \caption{We plot the error $\textN{\hat \abs - \abs}_2$ using \eqref{eq:finite_sample_sim_estimator} (solid lines), respectively
    $\hat \abs = \ubs_1(\hat \Mmat_{J^*})$ (dashed lines), for several link functions. The right plots in each subplot shows $J^*$ that is chosen according
    to the rule \eqref{eq:choice_J_conditional_SIM}.
    We see that in all cases where $f$ is a nonlinear, monotonic function $\ubs_1(\hat \Mmat_{J^*})$ improves upon \eqref{eq:finite_sample_sim_estimator}, especially
    in scenarios with moderate noise levels $\sigma_{\zeta}$. In the other cases
    the two estimators achieve similar accuracy.}
    \label{fig:sim_estimator_results}
\end{figure*}

\paragraph{Numerical setup and parameter choice.}
We sample $X \sim \CN(\boldsymbol{0},\Id_D)$, with $D=10$, and let $\abs = (1,0,\ldots,0)^\top$ (the specific choice of $\abs$ is irrelevant for the results due to the rotational invariance of $X$).
Responses are generated by $Y = f(\abs^\top X) + \zeta$, where $\zeta \sim \CN({0},\sigma_{\zeta}^2\textrm{Var}(f(\abs^\top X)))$.
We set $\alpha = 0.05$ and  additionally exclude subsets with $\SN{\CS_{J,\ell}} < 2D$, for the sake of numerical stability.
Our goal is twofold.
First, to compare estimation of the index vector using the standard OLS approach \eqref{eq:finite_sample_sim_estimator} with the strategy proposed in this section, and second, to empirically confer the observations described after Theorem \ref{thm:conditional_sim_estimator}.

A critical step of the modified approach is the selection of $J$ as a function of $N$.
As mentioned in Theorem \ref{thm:conditional_sim_estimator}, the denominator effectively
lower bounds $\lambda_1(\hat \Mmat_J)$, and the corresponding concentration bound can be written as
\begin{align*}
\textN{s \ubs_1(\hat \Mmat_J) - \abs}_{2}^2 \lesssim \frac{\rank(\VSigma) + \log(J) + u}{\alpha NJ} \frac{\sum_{\ell \in I_J}\hat \rho_{J,\ell}\kappa_{J,\ell} \textN{\Qmat\VSigma_{J,\ell}^{\dagger}\centerRV{X}_{J,\ell}}_{\psi_2}^2}{\lambda_1(\hat \Mmat_J)}.
\end{align*}
Neglecting the dependence of $\kappa_{J,\ell}$ and $\textN{\Qmat\VSigma_{J,\ell}^{\dagger}\centerRV{X}_{J,\ell}}_{\psi_2}^2$ on $J$
and $\log(J)$-terms, originating from union bounds in the proofs, the error is minimized
when maximizing $J\lambda_1(\hat \Mmat_J)$.
We thus propose to adaptively choose $J$ by
\begin{equation}
\label{eq:choice_J_conditional_SIM}
J^* := \max\{J=\lceil (1.5)^k\rceil : k \in \bbN_0,\  J\lambda_1(\hat \Mmat_J) > J' \lambda_1(\hat \Mmat_{J'})\,\,  \forall\,\,   J' < J\},
\end{equation}
where we use an exponential grid in $\bbN$ to decrease the computational demand. Note also that
$\lambda_1(\hat \Mmat_J) \geq \sum_{\ell \in I_{J}}\hat \rho_{J,\ell} \textN{\hat b_{J,\ell}}_2^2 \approx  \sum_{\ell \in I_{J}}\hat \rho_{J,\ell} \textN{b_{J,\ell}}_2^2$,
so, using parameter choice \eqref{eq:choice_J_conditional_SIM}, we expect no further increase of $J$ if $\textN{b_{J,\ell}}_2$'s decay rapidly.

\paragraph{Numerical experiments.}

Figure \ref{fig:sim_estimator_results} shows the errors $\textN{\hat \dbs - \abs}_2$ and the corresponding optimized choices $J^\ast$ of the number of level-sets $J$, for different link functions $f$.
Solid lines correspond to the estimator \eqref{eq:finite_sample_sim_estimator}, the dashed lines to $\ubs_1(\hat \Mmat_{J^*})$, and different colors represent different noise levels $\sigma_{\zeta}$.
In Figures \ref{fig:sim_estimator_logit} - \ref{fig:sim_estimator_tanh} we consider monotonic link functions and we can see that
the approach presented in this section performs substantially better than the standard OLS approach.
On the other hand, in Figures \ref{fig:sim_estimator_identity}, \ref{fig:absolute_1}
and \ref{fig:weird_function_1} the two approaches achieve similar performance.
In case of \ref{fig:sim_estimator_identity} this is because \eqref{eq:finite_sample_sim_estimator} is indeed optimal,
according to the Gauss-Markov Theorem, whereas link functions in \ref{fig:absolute_1}, \ref{fig:weird_function_1},
are not monotonic.
In the latter case, the plots of $J^*$ confirm that $\lambda_1(\hat \Mmat_J)$ decays rapidly as a function of $J$, leaving $J^*$ essentially constant
as a function of $N$.

Let us examine the results for monotonic functions in more detail.
The plots for $J^\ast$ in Figures \ref{fig:sim_estimator_logit} - \ref{fig:sim_estimator_tanh} show that the
number of level-sets $J^*$ indeed grows as a function of $N$, and it does so up to a level dictated
by the noise level $\sigma_{\zeta}$.
This shows that $\lambda_1(\hat \Mmat_J)$ does not decay faster than $J^{-1}$ over a reasonably large range of $N$ values, i.e. until  $ J^{-1}\approx\sigma_{\zeta}$.
As a consequence, our approach achieves a faster, $N^{-1}$ estimation rate for the index vector.
Specifically, in the noise-free case this holds asymptotically as $N\rightarrow \infty$.
On the other hand, in the case of corrupted $Y$'s, we first have a faster $N^{-1}$ convergence,
and then a sharp transition into the usual $N^{-1/2}$ rate.
The number of points $N$ at which this transition occurs depends inversely on the level of noise.

\section*{Acknowledgements}
We like to thank the anonymous referees for their feedback that helped to greatly improve earlier
versions of the manuscript.
This work is supported by the grant \emph{Function-driven Data Learning in High Dimensions}
by the Research Council of Norway.

\section{Appendix}
\subsection{Additional technical results}
\label{subsec:additional_results}
\begin{lemma}[Properties of the sub-Gaussian norm]
\label{lem:subGaussian_properties}
Let $X,Z$ be sub-Gaussian random vectors in $\bbR^D$, let $\VSigma = \Covv{X}$, and $\Amat \in \bbR^{D\times D}$.
Then we have
\begin{enumerate}[label =(\arabic*)]
\item\label{pro:centering} $\textN{X - \bbE X}_{\psi_2} \lesssim \textN{X}_{\psi_2}$ (holds also for sub-exponential random variables),
\item\label{pro:induced_norm} $\textN{\Amat\tilde X}_{\psi_2}\leq \textN{\Amat}_2 \textN{\Pmat_{\Im(\Amat^\top)}\tilde X}_{\psi_2}$,
\item\label{pro:cross_cov_psi2} $\textN{\Amat\VSigma \Amat^\top}_2 \lesssim \textN{\Amat\tilde X}_{\psi_2}^{2}$,
\item\label{pro:sube_subg_cs} $X^\top Z$ is sub-exponential with $\textN{\tilde X^\top \tilde Z - \bbE[\tilde X^\top \tilde Z]}_{\psi_1} \lesssim \textN{\tilde X}_{\psi_2}\textN{\tilde Z}_{\psi_2}$.
\end{enumerate}
\end{lemma}
\begin{proof}

\noindent
Property \ref{pro:centering} is shown in \cite[Lemma 2.6.8]{vershynin2018high} for sub-Gaussian random variables.
Applying the definition of the sub-Gaussian norm \eqref{eqn:subG_vect_norm} the claim follows,
since ${\vbs}^\top{X}$ is a sub-Gaussian random variable for every $\vbs\in\bbS^{D-1}$.
The same line of arguments holds for sub-exponential vectors.

For \ref{pro:induced_norm} we compute for arbitrary $\vbs \in \bbS^{D-1}$ and $\Amat^\top \vbs \in \Im(\Amat^\top)$
\begin{align*}
\textN{\vbs^\top \Amat \tilde X}_{\psi_2} &= \textN{\Amat^\top \vbs}_2 \textN{\left(\Amat^\top \vbs/\textN{\Amat^\top \vbs}_2\right)^\top\tilde X}_{\psi_2}
\\&\leq \textN{\Amat^\top}_2 \sup_{\ubs \in \Im(\Amat^\top) \cap \bbS^{D-1}}\textN{\ubs^\top\tilde X}_{\psi_2} \leq \textN{\Amat}_2 \textN{\Pmat_{\Im(\Amat^\top)}\tilde X}_{\psi_2}.
\end{align*}

For \ref{pro:cross_cov_psi2} we first note that \cite[Proposition 2.5.2]{vershynin2018high} implies
$\Var{u} \lesssim \textN{\tilde u}_{\psi_2}^2$ for any sub-Gaussian $u$, where $\tilde u=u-\bbE[u]$.
Thus, for every $\vbs \in \bbS^{D-1}$ we have
\[ \vbs^\top \Covv{\Amat X}\vbs = \Var{\vbs^\top \Amat X} \lesssim \textN{\vbs^\top \Amat\tilde X}_{\psi_2}^2.\]
Taking the supremum with respect to $\vbs \in \bbS^{D-1}$, the result follows since $\textCovv{\Amat X}$ is positive semidefinite.

Property \ref{pro:sube_subg_cs} follows from the centering property \ref{pro:centering} and \cite[Lemma 2.7.6]{vershynin2018high}.
\end{proof}

\begin{lemma}
\label{lem:sem_pos_def_quadratic_form}
Let $\Amat\in\bbR^{D\times D}$ be positive semidefinite, $\Bmat_1\in\bbR^{d_1\times D},\, \Bmat_2\in\bbR^{d_2\times D}$.
For $\ubs\in\bbR^{d_1},\,\vbs\in\bbR^{d_2}$ we have $\ubs^\top \Bmat_1\Amat \Bmat_2^\top \vbs \leq \sqrt{\ubs^\top \Bmat_1 \Amat\Bmat_1^\top \ubs}\sqrt{\vbs^\top \Bmat_2 \Amat\Bmat_2^\top \vbs}$.
 Moreover, $\textN{\Bmat_1 \Amat\Bmat_2}_2^2\leq \textN{\Bmat_1\Amat\Bmat_1^\top}_2\textN{\Bmat_2 \Amat\Bmat_2^\top}_2$.
\end{lemma}
\begin{proof}
Applying the Cauchy-Schwartz inequality we have
\begin{align}
\label{eq:aux_result_1}
\ubs^\top \Bmat_1 \Amat\Bmat_2^\top \vbs = \langle{\Amat^{1/2}\Bmat_1^\top \ubs},{ \Amat^{1/2}\Bmat_2^\top \vbs}\rangle\leq \textN{\Amat^{1/2}\Bmat_1^\top \ubs}_2\textN{\Amat^{1/2}\Bmat_2^\top \vbs}_2.
\end{align}
By the same line of argument we have
$\textN{\Amat^{1/2}\Bmat_1^\top \ubs}_2^2 = \ubs^\top \Bmat_1 \Amat\Bmat_1^\top \ubs$
and $\textN{\Amat^{1/2}\Bmat_2^\top \vbs}_2^2 = \vbs^\top \Bmat_2 \Amat\Bmat_2^\top \vbs$, giving the first statement.
Considering now $\textN{\ubs}_2=\N{\vbs}_2=1$, we have
\begin{align}\label{eqn:CS_term1}
\sup_{\substack{\N{\ubs}_2=1,\\ \N{\vbs}_2=1}} \ubs^\top \Bmat_1\Amat \Bmat_2^\top \vbs \leq \sqrt{\sup_{\N{\ubs}_2=1}\ubs^\top \Bmat_1 \Amat\Bmat_1^\top \ubs}\sqrt{\sup_{\N{\vbs}_2=1} \vbs^\top \Bmat_2 \Amat\Bmat_2^\top \vbs}.
\end{align}
Notice that since $\Amat$ is positive semidefinite, then $\Bmat_1 \Amat\Bmat_1^\top$ and $\Bmat_2 \Amat\Bmat_2^\top$ are positive semidefinite. Therefore,
\[\sup_{\N{\ubs}_2=1}\ubs^\top \Bmat_1 \Amat\Bmat_1^\top \ubs=\sup_{\N{\ubs}_2=1}\textN{(\Bmat_1 \Amat\Bmat_1^\top)^{1/2} \ubs}_2^2=\N{\Bmat_1 \Amat\Bmat_1^\top}_2.\]
An analogous expression holds for the other term. Identifying the quadratic form on the left hand side in \eqref{eqn:CS_term1} as the operator norm of $\Bmat_1\Amat\Bmat_2$, the conclusion follows.
\end{proof}

The following is a standard concentration bound for sub-Gaussian and sub-exponential random vectors around their mean.
\begin{lemma}
\label{lem:subExp_mean_estimation_improved}
Let $\{X_i : i \in [N]\}$ be independent copies of a centered random vector $X \in \bbR^D$.
Denote $\hat \Vmu := N^{-1}\sum_{i=1}^{N}X_i$ and $\msf(t) = t \vee t^2$.
For any $u>0$ the following hold with probability at least $1-\exp(-u)$.
\begin{enumerate}[label =(\arabic*),leftmargin=\widthof{[(2)]}]
\item If $\textN{X}_{\psi_2} < \infty$ we have
\(
\textN{\hat \Vmu}_2 \lesssim \textN{X}_{\psi_2}\sqrt{\frac{D+u}{N}}.
\)
\item If $\textN{X}_{\psi_1} < \infty$ we have
\(
\textN{\hat \Vmu}_2 \lesssim \textN{X}_{\psi_1}\msf\left(\sqrt{\frac{D+u}{N}}\right).
\)
\end{enumerate}
\end{lemma}
\begin{proof}
The argument for the two bounds follows along analogous lines. Let $\delta < 1/4$, and $\CN$ be a $\delta$-net
of $\bbS^{D-1}$. We first use \cite[Exercise 4.4.3]{vershynin2018high} to rewrite
\[
\N{\hat \Vmu}_2 = \sup_{\vbs \in \bbS^{D-1}} \vbs^\top \hat \Vmu = \sup_{\vbs \in \bbS^{D-1}} N^{-1}\sum_{i=1}^{N}\vbs^\top X_i
 \leq 2 \sup_{\vbs \in \CN} N^{-1}\sum_{i=1}^{N}\vbs^\top X_i.
\]
The term $N^{-1}\sum_{i=1}^{N}\vbs^\top X_i$ is a sum of either sub-Gaussian or sub-exponential random variables, for any $\vbs\! \in\! \CN$.
In the former case we have  $\textN{\vbs^\top X_i}_{\psi_{1}}\!\leq\!\textN{X}_{\psi_{1}}$, and $\textN{\vbs^\top X_i}_{\psi_{2}}\!\leq\!\textN{X}_{\psi_{2}}$ in the latter.
Hoeffding's inequality \cite[Theorem 2.6.2]{vershynin2018high}, in the sub-Gaussian case, or Bernstein's inequality
\cite[Theorem 2.8.1]{vershynin2018high}, in the sub-exponential case, now yield
concentration bounds for the sums.
The claim follows by applying the union bound over $\vbs\! \in\! \CN$, where the number of events is bounded by $\SN{\CN} \leq 12^D\!$,
see for instance \cite[Corollary 4.2.13]{vershynin2018high}.
\end{proof}

\subsection{Proofs for Section 2}
\label{subsec:proofs_sec_2}
\begin{proof}[Proof of Lemma \ref{lem:directed_covariance_estimation}]
Let $\centerRV{X} = X - \bbE X$.
Let $\Umat_{\Amat}\in\bbR^{d_A\times D}$ and $\Umat_{\Bmat}\in\bbR^{d_B\times D}$ be matrices whose rows contain the orthonormal basis for $\Im(\Amat\VSigma)$ and $\Im(\Bmat\VSigma)$, respectively.
Since $\VSigma$ and $\hat\VSigma$ are symmetric, and $\Im(\hat\VSigma)\subseteq \Im(\VSigma)$, we have $\hat\VSigma-\VSigma=\PrSigma(\hat\VSigma-\VSigma)\PrSigma$, where $\PrSigma$ is the orthogonal projection onto $\Im(\VSigma)$.
We thus have
\[ \textN{\Amat(\hat\VSigma-\VSigma)\Bmat^\top}_2 = \textN{\Amod(\hat\VSigma-\VSigma)\Bmod^\top}_2,\]
for $\Amod=\Umat_{\Amat}\Amat\PrSigma\in\bbR^{d_A\times D}$ and $\Bmod=\Umat_{\Bmat}\Bmat\PrSigma\in\bbR^{d_B\times D}$.
Denote now $\tilde \VSigma = N^{-1}\sum_{i=1}^{N}\centerRV{X_i} \centerRV{X_i}^\top$. Notice that $\tilde \VSigma$, compared to the empirical covariance $\hat \VSigma$, uses the true, instead of the empirical mean of $\centerRV{X}$.
We then have
$$
\textN{\Amod(\hat \VSigma - \VSigma)\Bmod^\top}_2 \leq
\textN{\Amod(\tilde \VSigma - \VSigma)\Bmod^\top}_2 + \Big\|{\sum_{i=1}^{N}\frac{\Amod\centerRV{X_i}}{N}}\Big\|_2\Big\|{\sum_{i=1}^{N}\frac{\Bmod\centerRV{X_i}}{N}}\Big\|_2.
$$
By Lemma \ref{lem:subExp_mean_estimation_improved} the second term
is always of higher order, and the resulting error can be absorbed into a corresponding upper
bound for the first term. Thus, in the following we focus on $\textN{\Amod(\tilde \VSigma - \VSigma)\Bmod^\top}_2$.

\noindent
We closely follow the proof of \cite[Proposition 2.1]{vershynin2012close}.
Let $\delta < 1/4$, and by $\CN$, $\CM$ denote $\delta$-nets of
 $\bbS^{d_A-1}$ and $\bbS^{d_B-1}$. From \cite[Exercise 4.4.3]{vershynin2018high} we have
\begin{align}
\label{eq:net_approximation}
\|{\Amod(\tilde \VSigma - \VSigma) \Bmod^\top}\|_2 &\leq (1-2\delta)^{-1} \sup\limits_{\substack{\xbs \in \CN\\ \ybs \in \CM}}\big\langle \Amod(\tilde \VSigma - \VSigma) \Bmod^\top \xbs, \ybs\big\rangle \\&\leq
2\sup\limits_{\substack{\xbs \in \CN\\ \ybs \in \CM}}\big\langle (\tilde \VSigma - \VSigma)  \Bmod^\top \xbs, \Amod^\top \ybs\big\rangle.
\end{align}
Consider now any pair $(\xbs,\ybs)\in\CN\times\CM$, and write
\begin{align*}
\big\langle{\tilde \VSigma \Bmod^\top \xbs},{\Amod^\top \ybs}\big\rangle = \sum_{i=1}^N \frac{\big\langle{\centerRV{X}_i (\centerRV{X}_i^\top\Bmod^\top \xbs)},{\Amod^\top \ybs}\big\rangle}{N} =
\sum\limits_{i=1}^{N}\frac{\big\langle{\Amod\centerRV{X_i}},{\ybs}\big\rangle \big\langle{\Bmod\centerRV{X_i}},{\xbs}\big\rangle}{N}.
\end{align*}
Since $\langle\Amod\centerRV{X_i},\ybs\rangle$ and $\langle\Bmod\centerRV{X_i},\xbs\rangle$ are sub-Gaussian, their product is sub-exponential, and from
\cite[Lemma 2.7.7]{vershynin2018high} we have
\begin{align*}
\|{\langle \Amod \centerRV{X_i}, \ybs\rangle \langle \Bmod \centerRV{X_i},\xbs\rangle}\|_{\psi_1}\! &\leq\!
\|{\langle \Amod \centerRV{X_i}, \ybs\rangle}\|_{\psi_2}\|{\langle \Bmod \centerRV{X_i},\xbs\rangle}\|_{\psi_2}\!\leq\! \|{\Amod\tilde X}\|_{\psi_2} \|{\Bmod\tilde X}\|_{\psi_2},
\end{align*}
which we denote by $\sigma_{\Amat \Bmat}:= \|{\Amod\tilde X}\|_{\psi_2} \|{\Bmod\tilde X}\|_{\psi_2}$ for short.
Since $\bbE[\tilde \VSigma]=\VSigma$, by Lemma \ref{lem:subExp_mean_estimation_improved}
we have for any $u > 0$, with $\msf(t) = t\vee t^2$,
\begin{equation*}
\bbP\left(\big\lvert{\big\langle \tilde \VSigma \Bmod^\top \xbs, \Amod^\top \ybs\big\rangle - \big\langle \VSigma \Bmod^\top \xbs, \Amod^\top \ybs\big\rangle}\big\rvert
\geq \sigma_{\Amat \Bmat} \msf\left(\sqrt{\frac{1+u}{N}}\right)\right) \leq \exp(-u).
\end{equation*}
The size of the nets can be bounded as $\SN{\CN} \leq 12^{d_A}$ and $\SN{\CM} \leq 12^{d_B}$, see \cite[Corollary 4.2.13]{vershynin2018high}.
Thus, considering all pair $(\xbs,\ybs)\in\CN\times\CM$ and using the union bound we get, for some universal constant $C$ (bounded e.g. by $\log(12) + 1$),
\begin{align*}
&\bbP\left({\Big\lvert{\sup\limits_{{\xbs \in \CN,\, \ybs \in \CM}}\big\langle{\tilde \VSigma \Bmod^\top \xbs},{\Amod^\top \ybs}\big\rangle\!-\!\big\langle{\VSigma \Bmod^\top \xbs},{\Amod^\top \ybs}}\big\rangle\Big\rvert \!\geq\! C \sigma_{\Amat \Bmat} \msf\left(\sqrt{\frac{u+(d_A + d_B)}{N}}\right)}\right) \\ &\qquad\qquad\qquad\qquad\qquad\leq  \SN{\CN}\SN{\CM}\exp\LRP{-u - (d_A + d_B)\log(12)}\\
&\qquad\qquad\qquad\qquad\qquad\leq \exp\left((d_A + d_B)\log(12)-u-(d_A + d_B)\log(12)\right).
\end{align*}

Using the $\delta$-net approximation bound \eqref{eq:net_approximation} we thus get
\begin{align*}
&\bbP\left(\|{\Amod(\tilde \VSigma - \VSigma) \Bmod^\top}\|_2\geq 2C \sigma_{\Amat \Bmat} \msf\left(\sqrt{\frac{u+(d_A + d_B)}{N}}\right)\right) \\ &\leq
\!\bbP\!\left({\Big\lvert{\sup\limits_{\substack{\xbs \in \CN\\ \ybs \in \CM}}\!\!\big\langle{\tilde \VSigma \Bmod^\top \xbs},{\Amod^\top \ybs}\big\rangle \!-\! \big\langle{\VSigma \Bmod^\top \xbs},{\Amod^\top \ybs}\big\rangle}\Big\rvert \! \geq\! C \sigma_{\Amat \Bmat} \msf\left(\sqrt{\frac{u+(d_A + d_B)}{N}}\right)}\right).
\end{align*}

The result now follows from $\textN{\Amod\centerRV{X}}_{\psi_2}=\textN{\Amat\centerRV{X}}_{\psi_2}$ and $\textN{\Bmod\centerRV{X}}_{\psi_2}=\textN{\Bmat\centerRV{X}}_{\psi_2}$
and thus $\sigma_{\Amat \Bmat} = \textN{\Amat\centerRV{X}}_{\psi_2}\textN{\Bmat\centerRV{X}}_{\psi_2}$.
\end{proof}

\begin{proof}[Proof of Lemma \ref{lem:concentration_sample_eigengap}]
We use the shorthand notation $\lambda_i := \lambda_{i}(\VSigma)$ and $\hat \lambda_i := \hat \lambda_{i}(\VSigma)$.
By the definition of $\delta_{il}$ we have
\begin{align*}
\delta_{il} &= \textSN{\hat \lambda_{i-1} - \lambda_{i}} \wedge \textSN{\lambda_{l} - \hat \lambda_{l+1}} \geq (\hat \lambda_{i-1} - \lambda_{i}) \wedge (\lambda_{l} - \hat \lambda_{l+1})\\
&\geq \left(\left(\lambda_{i-1}-\lambda_{i}\right) + (\hat \lambda_{i-1} - \lambda_{i-1})\right) \wedge \left(\left(\lambda_{l}-\lambda_{l+1}\right) + (\lambda_{l+1} - \hat \lambda_{l+1})\right),
\end{align*}
so it suffices to show $\hat \lambda_{i-1} - \lambda_{i-1}\geq -(\lambda_{i-1}-\lambda_{i})/2$
and $\lambda_{l+1} - \hat \lambda_{l+1} \geq - (\lambda_{l}-\lambda_{l+1})/2$.
Note that if $i=1$ or $l = \rank(\VSigma)$, the corresponding gap is trivial and
no concentration is required. Hence, we can focus on the case $i > 1$ and $l < \rank(\VSigma)$.
Using \cite[Proposition 3.13]{reiss2020nonasymptotic}, we obtain that for any $u>0$ a relative eigenvalue bound
\begin{align}
\label{eqn:aux_eigengap_1}
\hat\lambda_{i-1} - \lambda_{i-1} < -\frac{\lambda_{i-1}-\lambda_{i}}{2},
\end{align}
holds with probability at most $\exp(-u)$, provided
\begin{align}
\label{eqn:aux_N_eigengap}
N > C_K\left(\frac{\lambda_{i-1}}{\lambda_{i-1}-\lambda_{i}} \vee 1\right)\left(\sum_{j=1}^{i-1}\frac{\lambda_j}{\lambda_j - \lambda_{i-1} + \frac{\lambda_{i-1}-\lambda_{i}}{2}} \vee u\frac{\lambda_{i-1}}{\lambda_{i-1}-\lambda_{i}}\right).
\end{align}
Furthermore, since $\lambda_{i-1}-(\lambda_{i-1}-\lambda_{i})/2 < \lambda_j$ for all $j \geq i-1$,
the summands in the sum in \eqref{eqn:aux_N_eigengap} are increasing, and we thus get
\begin{align*}
\sum_{j=1}^{i-1}\frac{\lambda_j}{\lambda_j - \lambda_{i-1} + \frac{\lambda_{i-1}-\lambda_{i}}{2}} \leq (i-1)\frac{2\lambda_{i-1}}{\lambda_{i-1}-\lambda_{i}} \lesssim \rank(\VSigma) \frac{\lambda_{i-1}}{\lambda_{i-1}-\lambda_{i}}.
\end{align*}
Thus, \eqref{eqn:aux_N_eigengap} is implied by the requirement on $N$ in the statement. Similarly, using \cite[Proposition 3.10]{reiss2020nonasymptotic} we get that
\begin{align}
\label{eqn:aux_eigengap_2}
\hat\lambda_{l+1} - \lambda_{l+1} > \frac{\lambda_{l}-\lambda_{l+1}}{2}
\end{align}
holds with probability at most $\exp(-u)$, provided
\begin{align}
\label{eqn:aux_N_eigengap_2}
N > C_K\left(\frac{\lambda_{l+1}}{\lambda_{l}-\lambda_{l+1}} \vee 1\right)\left(\sum_{j=l+1}^{D}\frac{\lambda_j}{\lambda_{l+1} - \lambda_{j} + \frac{\lambda_{l}-\lambda_{l+1}}{2}} \vee u\frac{\lambda_{l+1}}{\lambda_{l}-\lambda_{l+1}}\right).
\end{align}
As before, we have $\lambda_j - (\lambda_{l}-\lambda_{l+1})/2 < \lambda_{l+1}$ for all $j \geq l+1$,
so that
\begin{align*}
\sum_{j=l+1}^{D}\frac{\lambda_j}{\lambda_{l+1} - \lambda_{j} + \frac{\lambda_{l}-\lambda_{l+1}}{2}} \leq (\rank(\VSigma)-l)\frac{2\lambda_{\ell+1}}{\lambda_{l}-\lambda_{l+1}} \lesssim \rank(\VSigma) \frac{\lambda_{l+1}}{\lambda_{l}-\lambda_{l+1}}.
\end{align*}
It follows that \eqref{eqn:aux_N_eigengap_2} is implied by the requirement on $N$ in the statement.
Taking the complementary event of the union of events \eqref{eqn:aux_eigengap_1}
and \eqref{eqn:aux_eigengap_2}, which holds with probability $1-2\exp(-u)$,
we get $\hat \lambda_{i-1} - \lambda_{i-1}\geq -(\lambda_{i-1}-\lambda_{i})/2$
and $\lambda_{l+1} - \hat \lambda_{l+1} \geq - (\lambda_{l}-\lambda_{l+1})/2$
and thus the result is proven.
\end{proof}

\subsection{Proofs for Section \ref{sec:precision_matrix_estimation}}
\label{subsec:proofs_sec_3}
\begin{lemma}
\label{lem:image_equality_sample}
If
$
\textN{\sqrt{\VSigma^{\dagger}} (\VSigma-\hat \VSigma) \sqrt{\VSigma^{\dagger}}}_2 < 1,
$
then $\ker(\hat \VSigma) \cap \Im(\VSigma) = \{0\}$.
\end{lemma}
\begin{proof}
We prove the statement by contraction.
Assume there exists a nonzero $\vbs\in \ker(\hat \VSigma) \cap \Im(\VSigma)$. Since $\vbs \in \Im(\VSigma)$ we have $\vbs \in \Im(\sqrt{\VSigma^{\dagger}})$ implying the existence of a $\ubs\in\bbS^{D-1}$ with $\vbs = \sqrt{\VSigma^{\dagger}}\ubs$ (possibly after scaling $\vbs$).  Using the properties of $\ubs$ and $\vbs$, we get a contradiction by
\begin{align*}
1 = \vbs^\top \VSigma \vbs = \vbs^\top  (\VSigma\!-\!\hat \VSigma)  \vbs = \ubs^\top\sqrt{\VSigma^{\dagger}}  (\VSigma-\hat \VSigma)  \sqrt{\VSigma^{\dagger}} \ubs \leq \big\|{\sqrt{\VSigma^{\dagger}}  (\VSigma-\hat \VSigma)  \sqrt{\VSigma^{\dagger}}}\big\|_2\!<\!1.
\end{align*}
\end{proof}

\begin{proof}[Proof of Theorem \ref{thm:condition_free_directional_precision_matrix_estimation_arbitrary}]
Assume for now $\textN{\sqrt{\VSigma^{\dagger}} (\VSigma-\hat \VSigma) \sqrt{\VSigma^{\dagger}}}_2 < \frac{1}{2}$ which will be
ensured later by concentration arguments. Conditioned on this event, Lemma \ref{lem:image_equality_sample} implies $\ker(\hat \VSigma) \cap \Im(\VSigma) = \{0\}$. Since $\Im(\hat \VSigma) \subset \Im(\VSigma)$ holds almost surely, this implies $\Im(\VSigma) = \Im(\hat \VSigma)$ almost surely.
Denote now $\VDelta := \hat \VSigma - \VSigma$.
We want to use a von Neumann series argument to express $\hat \VSigma^{\dagger}$ in terms
of $\VSigma^{\dagger}$ and the corresponding finite sample perturbation. Recall that for general symmetric matrices $\Umat$ and $\Vmat$ with  $\Im(\Umat) = \Im(\Vmat)$, we have  $(\Umat\Vmat)^{\dagger} = \Vmat^{\dagger}\Umat^{\dagger}$  \cite[Corollary 1.2]{tian2004some}. Letting
$\PrSigma$ be the orthogonal projection onto $\Im(\VSigma)$, and $\QrSigma=\Id_D-\PrSigma$, we can thus write
\begin{align*}
\hat\VSigma^{\dagger} &= (\VSigma + \Delta)^{\dagger} = \left(\VSigma^{1/2}(\PrSigma + \sqrt{\VSigma^{\dagger}}\Delta\sqrt{\VSigma^{\dagger}})\VSigma^{1/2}\right)^{\dagger} \\
&=\sqrt{\VSigma^{\dagger}}(\PrSigma + \sqrt{\VSigma^{\dagger}}\Delta\sqrt{\VSigma^{\dagger}})^{\dagger}\sqrt{\VSigma^{\dagger}}\\
&=\sqrt{\VSigma^{\dagger}}(\Id_D - \QrSigma +  \sqrt{\VSigma^{\dagger}}\Delta\sqrt{\VSigma^{\dagger}})^{\dagger}\sqrt{\VSigma^{\dagger}}\\
&=\sqrt{\VSigma^{\dagger}}(\Id_D +  \sqrt{\VSigma^{\dagger}}\Delta\sqrt{\VSigma^{\dagger}})^{-1}\sqrt{\VSigma^{\dagger}},
\end{align*}
where we used
the additivity of the Moore-Penrose inverse for matrices with orthogonal non-trivial eigenspaces in the last equality. Using a von Neumann series expansion for the second
term, we obtain
\begin{align*}
\hat\VSigma^{\dagger} &= \sqrt{\VSigma^{\dagger}}(\Id_D +  \sqrt{\VSigma^{\dagger}}\Delta\sqrt{\VSigma^{\dagger}})^{\dagger}\sqrt{\VSigma^{\dagger}} =\sqrt{\VSigma^{\dagger}}\sum_{k=0}^{\infty}(-\sqrt{\VSigma^{\dagger}}\Delta\sqrt{\VSigma^{\dagger}})^{k}\sqrt{\VSigma^{\dagger}} \\
&= \VSigma^{\dagger} + \sqrt{\VSigma^{\dagger}}\sum_{k=1}^{\infty}(-\sqrt{\VSigma^{\dagger}}\Delta\sqrt{\VSigma^{\dagger}})^{k}\sqrt{\VSigma^{\dagger}},
\end{align*}
where we used $\sqrt{\VSigma^{\dagger}}\sqrt{\VSigma^{\dagger}} = \VSigma^{\dagger}$ in the last equality.
Subtracting $\VSigma^{\dagger}$, and multiplying by $\Amat$ from the right and $\Bmat^\top$ from the left, it follows that
\begin{align}
\nonumber
&\Amat(\hat \VSigma^{\dagger} - \VSigma^{\dagger})\Bmat^\top = \Amat\sqrt{\VSigma^{\dagger}}\sum_{k=1}^{\infty}(-\sqrt{\VSigma^{\dagger}}\Delta\sqrt{\VSigma^{\dagger}})^{k}\sqrt{\VSigma^{\dagger}}\Bmat^\top\\
\label{eq:aux_identity_A_B}
&=  -\Amat\VSigma^{\dagger}\Delta \VSigma^{\dagger}\Bmat^\top + \Amat\VSigma^{\dagger}\Delta\sqrt{\VSigma^{\dagger}}\sum_{k=0}^{\infty}(-\sqrt{\VSigma^{\dagger}}\Delta\sqrt{\VSigma^{\dagger}})^{k}\sqrt{\VSigma^{\dagger}}\Delta \VSigma^{\dagger}\Bmat^\top.
\end{align}
The rest of the proof requires to apply Lemma \ref{lem:directed_covariance_estimation} multiple times to concentrate $\Delta$ when multiplied by different matrices from left and right. More specifically, since $\textN{\sqrt{\VSigma^{\dagger}}\centerRV{X}}_{\psi_2}^2 \gtrsim \textN{\textCovv{\sqrt{\VSigma^\dagger}X}}_2\gtrsim 1$,
we can assume the regime $N > \rank(\VSigma) + u$ by the assumption on $N$ in Theorem \ref{thm:condition_free_directional_precision_matrix_estimation_arbitrary}.
Then, we have by Lemma \ref{lem:directed_covariance_estimation} with probability $1-4\exp(-u)$
\begin{equation}
\label{eq:neumann_terms_bounded1}
\begin{aligned}
\|{\Amat\VSigma^{\dagger}\VDelta \VSigma^{\dagger} \Bmat^\top}\|_2 &\lesssim \|{\Amat\VSigma^{\dagger}\centerRV{X}}\|_{\psi_2}\|{\Bmat\VSigma^{\dagger}\centerRV{X}}\|_{\psi_2}\sqrt{\frac{\rank(\VSigma)+u}{N}}, \\
\|{\Amat\VSigma^{\dagger}\VDelta \sqrt{\VSigma^{\dagger}}}\|_2 &\lesssim \|{\Amat\VSigma^{\dagger}\centerRV{X}}\|_{\psi_2}\|{\sqrt{\VSigma^{\dagger}}\centerRV{X}}\|_{\psi_2}\sqrt{\frac{\rank(\VSigma)+u}{N}},\\
\|{\sqrt{\VSigma^{\dagger}}\VDelta \VSigma^{\dagger}\Bmat^\top}\|_2 &\lesssim\frac{\|{\Bmat\VSigma^{\dagger}\centerRV{X}}\|_{\psi_2}}{\|{\sqrt{\VSigma^{\dagger}}\centerRV{X}}\|_{\psi_2}}, \text{ and }{\textN{\sqrt{\VSigma^{\dagger}}\VDelta \sqrt{\VSigma^{\dagger}}}} \leq \frac{1}{2},
\end{aligned}
\end{equation}
whenever $N > C(\rank(\VSigma) + u)\textN{\sqrt{\VSigma^{\dagger}}\centerRV{X}}_{\psi_2}^4$ (we used properties of the sub-Gaussian norm in Lemma \ref{lem:subGaussian_properties}
to get $\textN{\sqrt{\VSigma^{\dagger}}\centerRV{X}}_{\psi_2}^2 \geq C\textN{\textrm{Cov}(\sqrt{\VSigma^{\dagger}}X)}_2  = C$)
and thus $\textN{\sqrt{\VSigma^{\dagger}}\centerRV{X}}_{\psi_2}^2\lesssim \textN{\sqrt{\VSigma^{\dagger}}\centerRV{X}}_{\psi_2}^4$).
Using identity \eqref{eq:aux_identity_A_B}, combining the above bounds and adjusting the uniform constant $C$, the desired result holds with probability $1-\exp(-u)$.
\end{proof}

\begin{proof}[Proof of Theorem \ref{thm:bounded_everything}]
For the covariance estimation bound denote first $\tilde \VSigma = N^{-1}\sum_{i=1}^{N}\centerRV{X}_i \centerRV{X}_i^\top$, where $\centerRV{X}_i = X_i - \bbE X$,
and decompose the error, as in the proof of Lemma \ref{lem:directed_covariance_estimation} into
\begin{align*}
\|{\Amat(\hat \VSigma - \VSigma)\Bmat^\top}\|_F&= \|{\Amod(\hat \VSigma - \VSigma)\Bmod^\top}\|_F \\&\leq \|{\Amod(\tilde \VSigma - \VSigma)\Bmod^\top}\|_F + \Big\|{\frac{1}{N}\sum_{i=1}^{N}\Amod\centerRV{X}_i}\Big\|_F\Big\|{\frac{1}{N}\sum_{i=1}^{N}\Bmod\centerRV{X}_i}\Big\|_F.
\end{align*}
The second term is again of higher order, with high probability, and can thus be disregarded. For the first term, define random matrices $\Vxi_i=\Amod\centerRV{X}_i\centerRV{X}_i^\top\Bmod^\top-\Amod\VSigma\Bmod^\top$.
First note $\bbE[\Vxi_i]=\boldsymbol{0}$, and
\begin{align*}
\textN{\Vxi_i}_{F} &\leq \textN{\Amod\centerRV{X}_i\centerRV{X}_i^\top\Bmod^\top}_F +\textN{\Amod^\top\VSigma\Bmod^\top}_F\\&=\textN{\Amod\centerRV{X}_i}_2\textN{\Bmod\centerRV{X}_i}_2+\bbE{{\textN{\Amod\centerRV{X}\centerRV{X}^\top\Bmod^\top}_F}}\leq2C_AC_B.
\end{align*}
Moreover, $\Vxi_i$ take values in the Hilbert space $(\bbR^{d_A\times d_B}, \textN{\cdot}_F)$.
Thus, by \cite[Section 2.4]{rosasco2010learningwithintegraloperator}, for every $u>0$ we have
\[ \bigg\| \frac{1}{N}\sum_{i=1}^N \Vxi_i\bigg\|_F\leq \frac{2\sqrt{2}C_AC_Bu}{\sqrt{N}},\]
with probability at least $1-\exp(-u^2)$.
The conclusion follows after adjusting the confidence level $u$ to account for the higher order term and the union bound.

The first steps for establishing the bound for precision matrices are the same as in the proof of Theorem \ref{thm:condition_free_directional_precision_matrix_estimation_arbitrary}.
Indeed, the only differences are in the following lines of inequalities.
Instead of \eqref{eq:neumann_terms_bounded1}, we have
\begin{equation*}
\begin{aligned}
\big\|{\Amat\VSigma^{\dagger}\VDelta \VSigma^{\dagger} \Bmat^\top}\big\|_F &\lesssim C_A^\dagger C_B^\dagger\sqrt{\frac{1+u}{N}}, \text{ and } \big\|{\sqrt{\VSigma^{\dagger}}\VDelta \VSigma^{\dagger}\Bmat^\top}\big\|_F \lesssim C_B^\dagger\sqrt{\Theta}\sqrt{\frac{1+u}{N}};\\
\big\|{\Amat {{{\VSigma^{\dagger}}}}\VDelta\sqrt{\VSigma^{\dagger}}}\big\|_F &\lesssim \frac{C_A^{\dagger}}{\sqrt{\Theta}}, \text{ and }
\big\|{\sqrt{\VSigma^{\dagger}}\VDelta \sqrt{\VSigma^{\dagger}}}\big\|_F\! \leq\! \frac{1}{2},
\end{aligned}
\end{equation*}
 using the just proven covariance bounds.
The claim now follows by using this together with \eqref{eq:aux_identity_A_B}, as in the proof of Theorem \ref{thm:condition_free_directional_precision_matrix_estimation_arbitrary}.
\end{proof}

\subsection{Proofs for Section \ref{sec:application}}
\label{subsec:proofs_sec_applications}
We first need a bound for $\rbs = \Covv{X,Y}$, and a concentration around the finite
sample estimate $\hat \rbs = N^{-1}\sum_{i=1}^{N}(X_i-\hat \Vmu_X)(Y_i - \hat \Vmu_Y)$.
\begin{lemma}
\label{lem:concentration_of_r}
{Let $\Amat\in\bbR^{k\times D}$.} If $X \in \bbR^{D},Y \in \bbR$ are sub-Gaussian, we have
$\textN{\Amat \rbs}_2 \lesssim \textN{\Amat \centerRV{X}}_{\psi_2}\textN{\centerRV{Y}}_{\psi_2}$.
Furthermore, with probability at least $1 - \exp(-u)$ we have
\begin{align*}
\N{\Amat (\rbs - \hat \rbs)}_2 \lesssim   \|{\Amat \centerRV{X}}\|_{\psi_2}\|{\centerRV{Y}}\|_{\psi_2}\LRP{\sqrt{\frac{k+u}{N}}\vee\frac{k+u}{N}}.
\end{align*}
\end{lemma}
\begin{proof}
By Lemma \ref{lem:subGaussian_properties} we have
$\Var{\vbs^\top \Amat  X} \lesssim \textN{\Amat  \centerRV{X}}_{\psi_2}^2$, and $\Var{Y}\lesssim \textN{\centerRV{Y}}_{\psi_2}^2$, which implies
\begin{align*}
\|{\Amat \rbs}\|_2 &= \|{\Covv{\Amat X,Y}}\|_2 = \sup_{\vbs \in \bbS^{k-1}} \vbs^\top \Covv{\Amat X,Y} =
\sup_{\vbs \in \bbS^{k-1}}  \Var{\vbs^\top \Amat  X,Y}\nonumber\\
&\leq \sup_{\vbs \in \bbS^{k-1}}  \sqrt{\Var{\vbs^\top \Amat  X}}\sqrt{\Var{Y}} \lesssim \|{\Amat \centerRV{X}}\|_{\psi_2}\|{\centerRV{Y}}\|_{\psi_2}.
\end{align*}
Define the random variable $Z_i := \centerRV{X_i}\centerRV{Y_i} - \Covv{X,Y}$ with $\bbE Z_i = 0$.
We write
\begin{align*}
\Amat (\rbs - \hat \rbs) = \frac{1}{N}\sum_{i=1}^{N}\Amat  Z_i - \Big(\frac{1}{N}\sum_{i=1}^{N}\Amat  \centerRV{X_i}\Big)\Big(\frac{1}{N}\sum_{i=1}^{N}\centerRV{Y_i}\Big).
\end{align*}
As in the proof of Lemma \ref{lem:directed_covariance_estimation}, by applying Lemma \ref{lem:subExp_mean_estimation_improved} it follows that the second term is of higher order.
On the other hand, the first term is an empirical mean of a sub-exponential centered variable $\Amat Z_i$,
\begin{align*}
\N{\Amat Z_i}_{\psi_1} &= \|{\Amat\centerRV{X_i}\centerRV{Y_i} - \Covv{\Amat X,Y}}\|_{\psi_1} \lesssim \|{\Amat\centerRV{X_i}\centerRV{Y_i}}\|_{\psi_1} \\&\lesssim \sup_{\vbs \in \bbS^{k-1}} \|{\vbs^\top \Amat\centerRV{X_i}\centerRV{Y_i}}\|_{\psi_1} \lesssim \|{\Amat\centerRV{X_i}}\|_{\psi_2}\|{\centerRV{Y_i}}\|_{\psi_2},
\end{align*}
where we use the centering property of the sub-exponential norm, and the bound for the sub-exponential norm by the
product of sub-Gaussian norms (Lemma \ref{lem:subGaussian_properties}).
Applying now Lemma \ref{lem:subExp_mean_estimation_improved}, the statement follows.
\end{proof}
\begin{proof}[Proof of Lemma \ref{lem:directional_linear_regression}]
We first note that $\textN{\sqrt{\VSigma^{\dagger}}X}_{\psi_2}^2 \gtrsim \textN{\textCovv{\sqrt{\VSigma^{\dagger}}X}}_2 \geq 1$,
so by choosing $C$ in the statement large enough, we assume the regime $N > (\rank(\VSigma) + u)$ in the following.
We begin with a bound for $\textN{\Pmat(\VSigma^{\dagger}\rbs - \hat \VSigma^{\dagger}\hat \rbs)}_2$.
\begin{align*}
&\|{\Pmat(\VSigma^{\dagger}\rbs - \hat \VSigma^{\dagger}\hat \rbs)}\|_2 \leq \|{\Pmat(\VSigma^{\dagger}\!-\! \hat \VSigma^{\dagger})\Pmat \rbs}\|_2\! +\! \|{\Pmat(\VSigma^{\dagger}\!-\! \hat \VSigma^{\dagger})\Qmat \rbs}\|_2\\
&\quad\quad\quad\quad\quad\quad\quad\quad\quad\quad+\! \|{\Pmat\hat \VSigma^{\dagger}\Pmat(\hat \rbs - \rbs)}\|_2
+ \|{\Pmat\hat \VSigma^{\dagger}\Qmat(\hat \rbs - \rbs)}\|_2\\
&\quad\leq \|{\Pmat(\VSigma^{\dagger}\!-\! \hat \VSigma^{\dagger})\Pmat}\|_2\|{\Pmat \rbs}\|_2
+ \|{\Pmat(\VSigma^{\dagger}\!-\! \hat \VSigma^{\dagger})\Qmat}\|_2\|{\Qmat \rbs}\|_2 \\&\qquad\qquad\qquad\qquad\qquad+ \|{\Pmat\hat \VSigma^{\dagger}\Pmat}\|_2\|{\Pmat(\hat \rbs - \rbs)}\|_2\! +\! \|{\Pmat\hat \VSigma^{\dagger}\Qmat}\|_2\|{\Qmat(\hat \rbs - \rbs)}\|_2.
\end{align*}
By Lemma \ref{lem:concentration_of_r} we have $\textN{\Qmat \rbs}_2 \leq \textN{\Qmat \centerRV{X}}_{\psi_2}\textN{\tilde Y}_{\psi_2}$, $\textN{\Pmat \rbs}_2 \leq \textN{\Pmat\centerRV{X}}_{\psi_2}\textN{\tilde Y}_{\psi_2}$.
Furthermore, since $\rbs \in \Im(\VSigma)$, we can rewrite $\textN{\Qmat(\rbs - \hat \rbs)}_2
= \textN{\ONBQSigma(\rbs - \hat \rbs)}_2$ and $\textN{\Pmat(\rbs - \hat \rbs)}_2
= \textN{\ONBPSigma(\rbs - \hat \rbs)}_2$, where the rows of $\ONBQSigma \in \bbR^{d_Q\times D}$ and
 $\ONBPSigma \in \bbR^{d_P\times D}$ contain orthonormal bases for $\Im(\Qmat\VSigma)$ and $\Im(\Pmat\VSigma)$, respectively.
 Using Lemma \ref{lem:concentration_of_r} and $d_Q \vee d_P \leq \rank(\VSigma)$,
we have with probability at least $1 - 2\exp(-u)$
\begin{equation}
\label{eq:aux_1_LR}
\begin{aligned}
\N{\Qmat(\rbs - \hat \rbs)}_2 \lesssim  \|{\Qmat\centerRV{X}}\|_{\psi_2}\|{\tilde Y}\|_{\psi_2} \sqrt{\frac{\rank(\VSigma) + u}{N}}, \\
\N{\Pmat(\rbs - \hat \rbs)}_2 \lesssim  \|{\Pmat\centerRV{X}}\|_{\psi_2}\|{\tilde Y}\|_{\psi_2} \sqrt{\frac{\rank(\VSigma) + u}{N}}.
\end{aligned}
\end{equation}
Furthermore, by Theorem \ref{thm:condition_free_directional_precision_matrix_estimation_arbitrary} whenever
$N > C (\rank(\VSigma)+u)\textN{\sqrt{\VSigma^{\dagger}}\centerRV{X}}_{\psi_2}^4$ we get with probability $1-2\exp(-u)$
for each $\Bmat \in \{\Pmat,\Qmat\}$
\begin{align}
\label{eq:aux_2_LR}
\|{\Pmat(\VSigma^{\dagger}\!-\! \hat \VSigma^{\dagger})\Bmat}\|_2 \lesssim \|{\Pmat\VSigma^{\dagger}\centerRV{X}}\|_{\psi_2}\|{\Bmat\VSigma^{\dagger} \centerRV{X}}\|_{\psi_2}\sqrt{\frac{\rank(\VSigma) + u}{N}}.
\end{align}
Conditioned on all four events we now have
\begin{align*}
\|{\Pmat(\VSigma^{\dagger} - \hat \VSigma^{\dagger})\Pmat}\|_2\|{\Pmat \rbs}\|_2
&\lesssim  \|{\tilde Y}\|_{\psi_2} \|{\Pmat\VSigma^{\dagger}\centerRV{X}}\|_{\psi_2} \sqrt{\kappa(\Pmat,X)}\sqrt{\frac{\rank(\VSigma) + u}{N}},\\
\|{\Pmat(\VSigma^{\dagger} - \hat \VSigma^{\dagger})\Qmat}\|_2\|{\Qmat \rbs}\|_2
&\lesssim  \|{\tilde Y}\|_{\psi_2} \|{\Pmat\VSigma^{\dagger}\centerRV{X}}\|_{\psi_2} \sqrt{\kappa(\Qmat,X)}\sqrt{\frac{\rank(\VSigma) + u}{N}}.
\end{align*}
Moreover, since $N > \rank(\VSigma) + u$ as stated in the beginning, we get
\begin{align*}
\|{\Pmat\hat \VSigma^{\dagger}\Pmat}\|_2\|{\Pmat(\hat \rbs - \rbs)}\|_2 &\leq \big({\|{\Pmat \VSigma^{\dagger}\Pmat}\|_2+ \|{\Pmat(\hat \VSigma^{\dagger}-\VSigma^{\dagger})\Pmat}\|_2}\big)\|{\Pmat(\hat \rbs - \rbs)}\|_2\\&\lesssim   \|{\centerRV{Y}}\|_{\psi_2}\|{\Pmat\VSigma^{\dagger}\centerRV{X}}\|_{\psi_2}  \sqrt{\kappa(\Pmat, X)}\sqrt{\frac{\rank(\VSigma) + u}{N}},
\end{align*}
and
\begin{align*}
\|{\Pmat\hat \VSigma^{\dagger}\Qmat}\|_2\|{\Qmat(\hat \rbs - \rbs)}\|_2 &\leq \big({\|{\Pmat \VSigma^{\dagger}\Qmat}\|_2+ \|{\Pmat(\hat \VSigma^{\dagger}-\VSigma^{\dagger})\Qmat}\|_2}\big)\|{\Qmat(\hat \rbs - \rbs)}\|_2\\&\lesssim \|{\centerRV{Y}}\|_{\psi_2}\|{\Pmat\VSigma^{\dagger}\centerRV{X}}\|_{\psi_2}\sqrt{\kappa(\Qmat, X)} \sqrt{\frac{\rank(\VSigma) + u}{N}},
\end{align*}
where we use property \ref{pro:cross_cov_psi2} in Lemma \ref{lem:subGaussian_properties}, exploiting the identity $\VSigma^\dagger=\VSigma^\dagger\VSigma\VSigma^\dagger$, and using Lemma \ref{lem:sem_pos_def_quadratic_form} in the second line.
Combining the bounds and conditioning on the events above, we have that with probability at least $1-4\exp(-u)$
\[
\|{\Pmat(\VSigma^{\dagger}\rbs - \hat \VSigma^{\dagger}\hat \rbs)}\|_2 \lesssim \| \centerRV{Y}\|_{\psi_2}\|
\Pmat\VSigma^{\dagger}\centerRV{X}\|_{\psi_2} \sqrt{\kappa(\Pmat,X) \vee \kappa(\Qmat,X)}\sqrt{\frac{\rank(\VSigma) + u}{N}},
\]
if $N> C (\rank(\VSigma) + u)(\textN{\sqrt{\VSigma^{\dagger}}\centerRV{X}}_{\psi_2}^4)$.
Adjusting the constant $C$ to account for the change in the probability constant and the requirement on $N$, the claim follows.
The proof for the bound on $\textN{\Qmat\hat \VSigma^{\dagger}\rbs}_2$ follows analogous lines or argument.
\end{proof}

\begin{proof}[Proof of Corollary \ref{cor:directional_normalized_linear_regression}]
Assume for the moment $\hat \bbs^\top \bbs > 0$. In this case $\bbs$ and $\Pmat\hat \bbs$ are co-linear since
$\Pmat\hat \bbs=\langle \dbs, \hat\bbs\rangle \dbs$. Therefore, $\dbs =\frac{\bbs}{\N{\bbs}_2} = \frac{\Pmat \hat \bbs}{\N{\Pmat\hat \bbs}_2} $,
which implies
\begin{align*}
\|{\Pmat(\hat \dbs - \dbs)}\|_2 =
\|{\Pmat\hat \bbs}\|_2 \bigg\lvert{\frac{1}{\|{\hat \bbs}\|_2} - \frac{1}{\|{\Pmat \hat \bbs}\|_2}}\bigg\lvert \leq
\bigg\lvert{\frac{\|{\Pmat\hat \bbs}\|_2 - \|{\hat \bbs}\|_2}{\|{\hat \bbs}\|_2}}\bigg\lvert \leq \frac{\|{\Qmat\hat \bbs}\|_2}{\|{\hat \bbs}\|_2}.
\end{align*}
Since $\Qmat \dbs=\boldsymbol{0}$ we now have
\[
\|{\hat \dbs - \dbs}\|_2 = \sqrt{\|{\Pmat(\hat \dbs - \dbs)}\|_2^2 + \|{\Qmat\hat \dbs}\|_2^2} \leq
\sqrt{2}\frac{\|{\Qmat\hat \bbs}\|_2}{\|{\hat \bbs}\|_2} \leq \sqrt{2}\frac{\|{\Qmat\hat \bbs}\|_2}{\|{\bbs}\|_2 - \|{\Pmat(\hat \bbs - \bbs)}\|_2}.
\]
Therefore, it suffices to ensure $\| \Pmat(\hat \bbs - \bbs)\|_2 < \frac{1}{2}\| \bbs\|_2$, which would also give $\hat \bbs^\top \bbs > 0$.
To that end, Lemma \ref{lem:directional_linear_regression} gives that $\| \Pmat(\hat \bbs - \bbs)\|_2 \leq \frac{1}{2} \| \bbs\|_2$ holds with probability at least $1-\exp(-u)$
provided
\[
N > C(\rank(\VSigma)+u)\bigg(\textN{\sqrt{\VSigma^{\dagger}}\centerRV{X}}_{\psi_2}^4 \vee \frac{\|{\centerRV{Y}}\|_{\psi_2}^2\|{\Pmat\VSigma^{\dagger}\centerRV{X}}\|_{\psi_2}^2\kappa_{\Pmat\Qmat}}{\|{\bbs}\|^2_2}\bigg).
\]
The result follows by bounding $\|\Qmat \hat \bbs\|_2$ through Lemma \ref{lem:directional_linear_regression}.
\end{proof}

\begin{proof}[Proof of Corollary \ref{cor:motivation_slicing_and_averaging}]
We first note \[\textN{\centerRV{Y}}_{\psi_2} \leq \textN{f(\abs^\top X) - \bbE f(\abs^\top X)}_{\psi_2} + \sigma_{\zeta}.\]
Furthermore, we can use a bounded difference inequality for unbounded spaces,
see \cite[Theorem 1]{kontorovich2014concentration}, to get
$\textN{f(\abs^\top X) - \bbE f(\abs^\top X)}_{\psi_2} \lesssim L\textN{\abs^\top X}_{\psi_2} \lesssim \sqrt{K} L\sigma_{\Pmat}$,
where we used the strict sub-Gaussianity in \eqref{eqn:norm_var_proxy_new} and $\abs^\top \VSigma \abs = \sigma_{\Pmat}^2$.
Therefore,
\begin{align*}
\|{\centerRV{Y}}\|_{\psi_2} &\lesssim KL\sigma_\Pmat + \sigma_{\zeta}.
\end{align*}
Using $\kappa(\Pmat,X) \lesssim K^2$, $\textN{\sqrt{\VSigma^{\dagger}}\centerRV{X}}_{\psi_2} = \sqrt{K}$, and plugging in the bounds on $\textN{\centerRV{Y}}_{\psi_2}$, as well as
\(\textN{\Pmat\VSigma^{\dagger}\centerRV{X}}_{\psi_2} = \frac{\sqrt{K}}{\sigma_\Pmat}\) and \(\textN{\Qmat\VSigma^{\dagger}\centerRV{X}}_{\psi_2} = \sqrt{K\textN{\Qmat\VSigma^{\dagger}\Qmat}}_2\),
the claim follows.
\end{proof}

\begin{lemma}\label{lem:subG_conditional}
If $Z$ is sub-Gaussian and $E$ an event with $\bbP(E)>0$, then $Z\vert E$ is also sub-Gaussian
\end{lemma}
\begin{proof}
Assume without loss of generality that $Z\in\bbR$. The argument for vectors then follows by the definition.
We use the characterization of sub-Gaussianity by the moment bound \cite[Proposition 2.5.2, b)]{vershynin2018high}.
Let $p\geq 1$. By the law of total expectation it follows
\[
\bbE[\lvert Z\rvert^p] = \bbE[\lvert Z\rvert^p\vert E]\bbP(E) +\bbE[\lvert Z\rvert^p\vert E^{\mathsf{c}}]\bbP(E^{\mathsf{c}})\geq \bbE[\lvert Z\rvert^p\vert E]\bbP(E).
 \]
 Dividing by $\bbP(E)$ and using the monotonicity of the $p$-th root yields
 \[ (\bbE[\lvert Z\rvert^p\vert E])^{1/p}\leq \frac{(\bbE[\lvert Z\lvert^p])^{1/p}}{\bbP(E)^{1/p}}\lesssim \frac{\N{Z}_{\psi_2}\sqrt{p}}{\bbP(E)},
 \]
where we used $\bbP(E)\leq 1$ and the sub-Gaussianity of $Z$ in the last inequality.
\end{proof}

\begin{proof}[Proof of Theorem \ref{thm:conditional_sim_estimator}]
Using $0 \leq \max_{s = \pm 1}\langle s \ubs_1(\hat \Mmat_J), \abs\rangle \leq 1$, we first compute
\begin{align*}
\min_{s =  \pm 1 }\|{s\ubs_1(\hat \Mmat_J) - \abs}\|_2^2 &= \min_{s = \pm 1 }2\big(1-\langle s \ubs_1(\hat \Mmat_J), \abs\rangle\big)
\\
&\leq 2\big(1-\langle \ubs_1(\hat \Mmat_J), \abs\rangle^2\big) \leq 2\|{\Qmat \ubs_1(\hat \Mmat_J)}\|_2^2.
\end{align*}
The Davis-Kahan Theorem \cite[Theorem 7.3.1]{bhatia2013matrix} for $\hat \Mmat_J$ and $\Pmat = \abs\abs^\top$ then gives
\begin{align}
\label{eq:aux_11}
\|{\Qmat \ubs_1(\hat \Mmat_J)}\|_2^2\! =\!\|{\Qmat \ubs_1(\hat \Mmat_J)\ubs_1(\hat \Mmat_J)^\top}\|_F^2 \leq \frac{\|{\Qmat(\Pmat - \hat \Mmat_J)}\|_F^2}{\lambda_1(\hat \Mmat_J)^2} =
\frac{\|{\Qmat\hat \Mmat_J}\|_F^2}{\lambda_1(\hat \Mmat_J)^2}.
\end{align}
{It remains to find a concentration bound for the matrix $\Qmat \hat \Mmat_J$ around zero and
a lower bound for the leading eigenvalue $\lambda_1(\hat \Mmat_J)$.}
Now let $\pi : [\SN{I_J}] \rightarrow I_J$ be bijective and introduce the matrix
\[
\hat \Bmat_J = \Big[\sqrt{\hat \rho_{J,\pi(1)}}\hat \bbs_{J,\pi(1)}  |\ldots|\sqrt{\hat \rho_{J,\pi(\SN{I_J})}}\hat \bbs_{J,\pi(\SN{I_J})}\Big] \in \bbR^{D\times \SN{I_J}},
\]
satisfying $\hat \Mmat_J = \hat \Bmat_J  \hat \Bmat_J^\top$, and thus $\lambda_1(\hat \Mmat_J) = \lambda_1(\hat \Bmat_J  \hat \Bmat_J^\top)$.
Using $\textN{\tempMatOne\tempMatTwo}_F \leq \textN{\tempMatOne}_2 \textN{\tempMatTwo}_F$ and $\textN{\tempMatOne}_F=\textN{\tempMatOne^\top}_F$, which hold for arbitrary $\tempMatOne$ and $\tempMatTwo$, yields
\begin{align*}
\|{\Qmat \hat \Mmat_J}\|_F^2 = \|{\Qmat\hat \Bmat_J  \hat \Bmat_J^\top }\|_F^2 \leq \|{\hat \Bmat_J}\|_2^2 \|{\Qmat \hat \Bmat_J}\|_F^2 \leq
\lambda_1(\hat \Mmat_J)\|{\Qmat \hat \Bmat_J}\|_F^2.
\end{align*}
By Lemma \ref{lem:subG_conditional}, conditioning on a sub-Gaussian random vector gives a sub-Gaussian random vector. Thus, by \eqref{eq:lin_reg_bounds_dir_Q} we have
simultaneously for all $\ell \in I_J$
\begin{align}
\label{eq:aux_bound_b}
\|{\Qmat \hat \bbs_{J,\ell}}\|_2^2\lesssim  \|{\centerRV{Y}_{J,\ell}}\|_{\psi_2}^2 \|{\Qmat\VSigma_{J,\ell}^\dagger \centerRV{X}_{J,\ell}}\|_{\psi_2}^2 \kappa_{J,\ell}\frac{\rank(\VSigma_{J,\ell}) +  \log(\SN{I_J}) + u}{N_{J,\ell}},
\end{align}
with probability at least $1 - \exp(-u)$, provided $N_{J,\ell} > C(\rank(\VSigma_{J,\ell}) + \log(\SN{I_J})+u)K_{J,\ell}^4$
for all $\ell \in I_J$.
By definition of $I_J$, we have $N_{J,\ell} > \alpha NJ^{-1}$, and thus the previous condition
is satisfied whenever for all $\ell \in I_J$
\begin{align*}
\frac{N}{J} &> \frac{C(\rank(\VSigma_{J,\ell}) + \log(J) + u)K_{J,\ell}^4}{\alpha},
\end{align*}
which is implied by $N > CJK_J^4(\rank(\VSigma)+\log(J)+u)/\alpha $.
Under the same conditions and with similar probability we obtain
\begin{align*}
\|{\Qmat \hat \Mmat_J}\|_F^2 &\leq \lambda_1(\hat \Mmat_J) \sum_{\ell \in I_J}\hat \rho_{J,\ell} \|{\Qmat \hat \bbs_{J,\ell}}\|_2^2
\\&\lesssim \lambda_1(\hat \Mmat_J) \sum_{\ell \in I_J}\hat \rho_{J,\ell}\|{\centerRV{Y}_{J,\ell}}\|_{\psi_2}^2\|{\Qmat\VSigma_{J,\ell}^\dagger \centerRV{X}_{J,\ell}}\|_{\psi_2}^2 \kappa_{J,\ell}\frac{\rank(\VSigma) + \log(J) + u}{N_{J,\ell}} \\
&\lesssim \frac{\lambda_1(\hat \Mmat_J)(\rank(\VSigma) +\log(J) + u)}{\alpha N J} \sum_{\ell \in I_J}\hat \rho_{J,\ell}\|{\Qmat\VSigma_{J,\ell}^\dagger \centerRV{X}_{J,\ell}}\|_{\psi_2}^2 \kappa_{J,\ell},
\end{align*}
where we used $N_{J,\ell} > \alpha N J^{-1}$ and  $\textN{\centerRV{Y}_{J,\ell}}_{\psi_2} \lesssim  \SN{\CR_{J,\ell}} = J^{-1}$
(definition of $\CR_{J,\ell}$) in the last inequality.
{To bound $\lambda_1(\hat \Mmat_J)$ from below, we note that on the same event where
bound \eqref{eq:aux_bound_b} holds, we also have for all $\ell \in I_J$
\begin{equation}
\begin{aligned}
\label{eq:aux_bound_b_2}
\|{\Pmat (\hat \bbs_{J,\ell} -\bbs_{J,\ell})}\|_2^2 &\lesssim \|{\centerRV{Y}_{J,\ell}}\|_{\psi_2}^2 \|{\Pmat\VSigma_{J,\ell}^\dagger \centerRV{X}_{J,\ell}}\|_{\psi_2}^2 \kappa_{J,\ell}\frac{\rank(\VSigma_{J,\ell}) + \log(J) + u)}{N_{J,\ell}}\\
&\lesssim \frac{\rank(\VSigma) + \log(J) + u}{\alpha N J}\|{\Pmat\VSigma_{J,\ell}^\dagger \centerRV{X}_{J,\ell}}\|_{\psi_2}^2 \kappa_{J,\ell}
\end{aligned}
\end{equation}
Thus, multiplying $\hat \Mmat_J$ from left and right by $\abs$ and using \ref{ass:alignedness}, we obtain
\begin{align*}
&\lambda_1(\hat \Mmat_J)\!\geq\!\abs^\top \hat \Mmat_J\abs\!=\!\sum_{\ell \in I_{J}}\hat \rho_{J,\ell}\langle \abs, \hat b_{J,\ell}\rangle^2
= \sum_{\ell \in I_{J}}\hat \rho_{J,\ell}\left(\N{b_{J,\ell}}_2^2 - \N{\Pmat(b_{J,\ell}-\hat b_{J,\ell})}_2^2\right)\\
&\qquad\geq\sum_{\ell \in I_{J}}\hat \rho_{J,\ell}\left(\N{b_{J,\ell}}_2^2 - \frac{C(\rank(\VSigma) + \log(J) + u)}{\alpha N J}\N{\Pmat\VSigma_{J,\ell}^\dagger \centerRV{X}_{J,\ell}}\|_{\psi_2}^2 \kappa_{J,\ell}\right)
\end{align*}
for some universal constant $C > 0$. The result follows by plugging the bounds on $\|{\Qmat \hat \Mmat_J}\|_F^2$
and $\lambda_1(\hat \Mmat_J)$ into \eqref{eq:aux_11}, and using the definition of $\varepsilon_{N,J,u}$.
}
\end{proof}

\end{document}

%% file: preamble.tex
\usepackage{latexsym, amsfonts, amsmath, amssymb, amsthm}
\usepackage[utf8]{inputenc}
\usepackage[final=true,protrusion=true,expansion=true]{microtype}
\usepackage{cite}
\usepackage{enumitem}
\usepackage{color}
\usepackage{mathtools}
\usepackage{mathrsfs}
\usepackage{upgreek}
\usepackage{longtable}
\usepackage{subfigure}
\usepackage{bbm}

\usepackage{booktabs}
\usepackage{multirow}

\usepackage[font=scriptsize,labelfont=bf]{caption}
\usepackage[affil-it]{authblk}

\usepackage{geometry}
\usepackage{hyperref}
\hypersetup{
	final,
    colorlinks=true,
    linkcolor=blue,
    filecolor=magenta,
    urlcolor=cyan,
}

\theoremstyle{plain}
\newtheorem{theorem}{Theorem}
\newtheorem{proposition}[theorem]{Proposition}
\newtheorem{lemma}[theorem]{Lemma}
\newtheorem{corollary}[theorem]{Corollary}

\newtheorem{example}[theorem]{Example}
\newtheorem{remark}[theorem]{Remark}